\documentclass[12pt,reqno]{amsart}
\usepackage{pgfplots}
\usepackage{tikz}
\pgfplotsset{compat=newest}
\usepackage{color}
\makeatletter
\renewcommand*{\@cite}[2]{\fcolorbox{black}{white}{#1\if@tempswa, #2\fi}}
\renewcommand*{\@biblabel}[1]{{\fcolorbox{green}{white}{#1}}\hfill}
\makeatother
\usepackage{amsmath,amsfonts,amssymb,amsthm}
\usepackage{graphicx}
\graphicspath{ }
\usetikzlibrary{intersections}
\usetikzlibrary{patterns}
\usepackage{soul}
\usepackage[title]{appendix}
\usepackage{lipsum}
\usepackage{epstopdf}
\usepackage{pdflscape}
\usepackage{csquotes}
\usepackage{wrapfig}
\usepackage{accents}
\usepackage{adjustbox}
\usepackage{tikz-3dplot}
\usepackage{caption}
\usepackage{subcaption}
\usepackage{calligra}
\usepackage[colorlinks]{hyperref}
\hypersetup{citecolor=black}

\numberwithin{figure}{section}

\def\AAA{\mathcal{A}}

\def\CCC{\mathcal{C}}
\def\DDD{\mathcal{D}}
\def\FF{\mathbb{F}}
\def\FFF{\overline{\mathbb{F}}}
\def\FFFF{\mathcal{F}}

\def\mm{\mathfrak{m}}

\def\NNN{\mathcal{N}}
\def\OO{\mathcal{O}}

\def\pp{\mathfrak{p}}
\def\Q{\mathbb{Q}}
\def\QQ{\overline{\mathbb{Q}}}
\def\R{\mathbb{R}}

\def\UUU{\mathcal{U}}

\def\Z{\mathbb{Z}}

\def\1{\mathbf{1}}

\theoremstyle{plain}
\newtheorem{thm}{Theorem}[section]
\newtheorem{lem}[thm]{Lemma}
\newtheorem{prop}[thm]{Proposition}
\newtheorem{cor}{Corollary}[thm]
\theoremstyle{definition}
\newtheorem{defn}[thm]{Definition}

\newtheorem{exmp}[thm]{Example}
\newtheorem{rem}[thm]{Remark}

\numberwithin{equation}{section}

\usepackage{mathtools}
\makeatletter
\@namedef{subjclassname@2020}{%
	\textup{2020} Mathematics Subject Classification}
\makeatother

\title{Constructing $2$-dimensional Lubin-Tate formal groups over $\mathbb{Z}_{p}$ (I)}
\author[R. Abdellatif]{Ramla Abdellatif$^{*}$} 
\address{$^{*}$Laboratoire Ami\'enois de Math\'ematique Fondamentale et Appliqu\'ee\\ Universit\'e de Picardie Jules Verne\\  33, rue Saint-Leu - 80 039 Amiens Cedex 1 - France.}
\email{ramla.abdellatif@math.cnrs.fr}
\author[M. A. Sarkar]{Mabud Ali Sarkar$^{\dagger}$}
\address{$^{\dagger}$Department of Mathematics\\ The University of Burdwan \\ Burdwan-713101, India.}
\email{mabudji@gmail.com}

\begin{document}

	\begin{abstract}
		 In this paper, we construct a class of $2$-dimensional formal groups over $\mathbb{Z}_p$ that provide a higher-dimensional analogue of the usual $1$-dimensional Lubin-Tate formal groups, then we initiate the study of the extensions generated by their $p^{n}$-torsion points. For instance, we prove that the coordinates of the $p^{\infty}$-torsion points of such a formal group generate an abelian extension over a certain unramified extension of $\mathbb{Q}_{p}$, and we study some ramification properties of these abelian extensions. In particular, we prove that the extension generated by the coordinates of the $p$-torsion points is in general totally ramified. 
	\end{abstract}
	\maketitle
	\tableofcontents
	\section{Introduction} 
	
Let $p$ be a prime integer. In 1964, Lubin and Tate \cite{JL2} constructed a one-dimensional formal group, now refered to as the {\it Lubin-Tate formal group over $\mathbb{Z}_{p}$}, and proved that the $p^{\infty}$-torsion points of this formal group generate a maximal totally ramified abelian extension of $\mathbb{Q}_p$ that allowed them to recover (the opposite of) the reciprocity law at the core of local class field theory.
More precisely, set 
	\begin{equation} \label{e1}
     	\CCC_{LT} :=\{f(x) \in x \mathbb{Z}_{p}[[x]] ~|~ f(x)=p x \ ( \text{mod deg} \ 2) \ \text{and} \ f(x) \equiv x^{p} \ (\text{mod} \ p)\}. 
	\end{equation} 
Then, given any element $f \in C_{LT}$, Lubin and Tate proved that there is a unique one-dimensional Lubin-Tate formal group $F(x,y)$ over $\mathbb{Z}_{p}$ such that $f(F(x,y))=F(f(x),f(y))$, and that the zeros of the endomorphism $[p^{n}]_{F}$ generate a totally ramified abelian extension of $\mathbb{Q}_{p}$.
     
     Several authors built more or less satisfying extensions of the work done by Lubin and Tate. To state them, let us fix an unramified extension $U$ of $\Q_{p}$, with ring of integers $\OO_{U}$ and maximal ideal $\pp$. For a fixed positive integer $d$, let $U^{nr}$ denote the maximal unramified extension of degree $d$ of $U$ and let $\varphi : U^{nr} \to U^{nr}$ denote the Frobenius automorphism of $U^{nr}$.
     \begin{itemize}
     \item In the 80's, Iwasawa \cite{KI} constructed a more general one-dimensional Lubin-Tate formal group $F$ that satisfies $f(F(x,y))=(\varphi_{*}F)(f(x),f(y))$ for any $f(x) \in \mathcal{C}_{LT}$, where $\varphi_*F$ is obtained by applying $\varphi$ to the coefficients of $F$. 
     \item Around the same time, De Shalit \cite{EDS} introduced a {\it relative Lubin-Tate formal group $\FFFF$}, associated with $U/\Q_{p}$, that produces a tower of abelian extensions $U(\FFFF[\pp^{n}])/\Q_{p}$ such that each extension $U(\FFFF[\pp^{n}])/U$ is totally ramified.
     \item Twenty years later, Berger \cite{LB} showed that the torsion points of the relative Lubin-Tate formal group defined by De Shalit generate a tower of totally ramified extensions of $\Q_{p}$.
     \end{itemize}

All these constructions, however, provide one-dimensional formal groups. In the 70's, Honda \cite{TH} and Hazewinkel \cite{MH} laid a path to build some higher-dimensional formal groups, but most attempts to follow this way were not pushed very far (see for instance the work of Koch \cite{IK}, whose constructions of higher dimensional Lubin-Tate groups actually amount to direct sums of one-dimensional Lubin-Tate groups) and the behaviour of higher dimensional Lubin-Tate groups still remains poorly understood. For the sake of completeness, let us also mention the work done by Nakamura \cite{TN} in positive characteristic, which provides some examples of $2$-dimensional formal groups over $\FF_{p}$ that are very close to $1$-dimensional formal groups too.

In 2020, Matson \cite{CM} extended Hazewinkel's work in her PhD Thesis: she developed there some suitable family of higher-dimensional formal groups with complex multiplication under \enquote{certain conditions},  then used them to generate abelian extensions over some unramified extension of any given $p$-adic field.\\
     
This paper aims at filling partially this gap in dimension $2$, by introducing {\it actual} two-dimensional Lubin-Tate groups over $\Z_{p}$ (in the sense that they are not isomorphic to a direct sum of one-dimensional Lubin-Tate groups) and initiating the study of the extensions generated by their $p^{\infty}$-torsion points. It is organised as follows.

 In Section \ref{s2}, we construct the aforementioned class of two-dimensional formal groups over $\mathbb{Z}_{p}$. We show that they have larger rings of endomorphisms than $\mathbb{Z}_{p}$ when considered as defined over certain unramified extensions, hence they are a suitable analogue of one-dimensional Lubin-Tate formal groups over $\mathbb{Z}_{p}$. 
 
 To provide a geometric understanding of our construction, we define in Section \ref{s3} the Newton copolygon associated with a $p$-adic power series in two variables. 
 
In Section \ref{s4}, we start to investigate the properties of the extensions generated by the $p^{n}$-torsion points of these formal groups. In particular, we prove that for such a formal group $F$ and any positive integer $n$, the coordinates of the $p^{n}$-torsion points of $F$ generate a Galois extension (Proposition \ref{p4.3}), and that the $p^{\infty}$-torsion points of $F$ generate an abelian extension over a certain unramified extension $U_{h}$ of $\Q_{p}$ (Theorem \ref{t4.2}). Finally, we show that when $h$ is odd, the $p$-torsion points of $F$ generate a totally abelian extension of $U_{h}$ (Theorem \ref{t4.4bis}).

 Two appendices are attached to this paper: Appendix \ref{sb} provides an alternative proof of Proposition \ref{p1.1} in the $2$-dimensional case, which is the one we are interested in in this paper, while Appendix \ref{sa} explains some limitations preventing a direct extension of our work beyond the $2$-dimensional case.

     \subsection*{Acknowledgements}  The authors are deeply grateful and indebt to Professor Jonathan Lubin for suggesting this problem and offering many helfpul feedbacks and suggestions throughout the preparation of this paper. They also express their warmest thanks to Professor Daniel Barsky for many insightful conversations regarding the study of ramifications in Theorem \ref{t4.4bis}.\\
     The first author is partially supported by the \text{Agence Nationale de la Recherche}, through the project CoLoSS (ANR-19-CE40-0015). The second author is supported by the \textit{Council Of Scientific and Industrial Research (CSIR)}, Government of India, through a Senior Research Fellowship.
      
     \section{Preliminary results and statements}  \label{s1}
          \subsection{General notations}  We fix a prime number $p$ and let $\mathbb{Q}_{p}$ denote the field of $p$-adic numbers, $\Z_{p}$ denote its ring of integers ({\it $p$-adic integers}), $\mm = p\Z_{p}$ denote its unique maximal ideal and $\FF_{p} = \Z_{p}/p\Z_{p} \simeq \Z/p\Z$ denote its residue field. We also fix an algebraic closure $\QQ_{p}$ of $\Q_{p}$, whose ring of integers is $\overline{\Z_{p}}$ and whose maximal ideal is $\overline{\mm}$. We let $v : \Q_{p}^{*} \to \Z$ denote the usual normalised $p$-adic valuation, as well as its extension to any finite extension of $\Q_{p}$ in $\overline{\Q}_{p}$.
          
          For any integer $k \geq 1$, we let $U_{k}$ denote the unramified extension of degree $k$ of $\Q_{p}$ in $\overline{\Q}_{p}$. The ring of integers of $U_{k}$ will be denoted by $\UUU_{k}$. 
          
          For any power $q$ of $p$, we denote (as usual) by $\FF_{q}$ the finite field with $q$ elements in a (once for all) fixed algebraic closure $\FFF_{p}$ of $\FF_{p}$.
          
          \subsection{Some remainder on formal groups} 
We follow the definitions and notations given in \cite{MH}. The reader familiar with this setting can skip this section and move forward to Section \ref{s2}.
	\begin{defn}
	Let $R$ be a commutative ring. Consider $n$-tuples of indeterminates $X=(x_{1},\cdots, x_{n})$, $Y=(y_{1}, \cdots, y_{n})$ and $Z = (z_{1}, \ldots , z_{n})$ over $R$. Also denote by $\vec{0} = (0, \ldots 0)$ the null $n$-tuple in $R^{n}$.\\
	An {\it $n$-dimensional formal group over $R$} is an $n$-tuple power series $F(X,Y)=(F_{1}(X,Y), \cdots, F_{n}(X,Y))$ in $R[[X,Y]]^{n}$ that satisfies the four following properties.
		\begin{enumerate}
		   \item[(i)] $F(X,Y)=X+Y \text{ mod degree 2 terms}$, i.e. $F_{i}(X,Y)=x_{i}+y_{i} \text{ mod  degree 2 terms in $x_{i}$ and $y_{i}$}$ for all $i \in \{1,2, \cdots, n\}$. 
		   \item[(ii)] Associativity : $F(F(X,Y),Z)=F(X,F(Y,Z))$.
		   \item[(iii)] Identity element : $F(X,\vec{0})=F(\vec{0},X)=X$.
		   \item[(iv)] Symmetry : There exists a unique power series $\iota(X) \in R[[X]]^n$ such that $F(X,\iota(X))= \vec{0} = F(\iota(X), X)$.
		\end{enumerate}
	\end{defn}
	Given an $n$-dimensional formal group $F$, we sometimes write $X +_{F} Y$ instead of $F(X,Y)$.
	\begin{defn}
	An $n$-dimensional formal group $F$ is said {\it commutative} when it satisfies 
	\[ X+_{F}Y = Y+_{F} X \ . \]
	\end{defn}
	\begin{exmp} \label{exFG}
		\begin{enumerate}
			\item[(i)] $G_{a}(X,Y) :=X+Y$ and $G_{m}(X,Y) :=X+Y+XY=(X+1)(Y+1)-1$ are both 1-dimensional commutative formal groups, respectively called the {\it additive} and {\it multiplicative} formal groups over $R$.  
			\item[(ii)] More generally, for any positive integer $n$, the {\it $n$-dimensional additive formal group} is defined by 
			\[ G^{(n)}_{a}(X,Y)=(G_{a,1}(X,Y), \cdots, G_{a,n}(X,Y)) \ , \] 
			where $G_{a,i}(X,Y)=x_{i}+y_{i}$ is the $1$-dimensional additive formal group defined in $(i)$ above, for all $i \in \{1, \ldots, n\}$. This formal group law is usually denoted by $X + Y$.
			\item[(iii)] One can check as an exercise that the $2$-tuple power series 
			\[ F(X,Y)=(F_1(X,Y),F_2(X,Y))= (x_{1}+y_{1}+x_{2}y_{2} \ ; \ x_{2}+y_{2} ) \] 
			defines a $2$-dimensional formal group over $R$.
	\end{enumerate}
	\end{exmp}
The next definitions come from \cite[Section 9.4]{MH} and provide the suitable notions of (iso)morphism of formal groups over a domain $R$. 
 \begin{defn} Given two positive integers $m$ and $n$, let $F(X,Y)$ be an $m$-dimensional formal group over $R$ and $G(X,Y)$ be an $n$-dimensional formal group over $R$.
 \begin{itemize} 
\item  A {\it homomorphism $f: F(X,Y) \to G(X,Y)$ over $R$} is an $n$-tuple power series in $m$ indeterminates $f(X)=(f_{1}(X), \cdots, f_{n}(X))$ with coefficients in $R$, such that $f(X) \equiv 0$ mod degree 1 terms and $f(X+_{F}Y)=f(X)+_{G}f(Y)$, i.e. $f(F(X,Y))=G(f(X),f(Y))$. 
\item Such a homomorphism is an {\it isomorphism over $R$} when there exists a homomorphism $g(X): G(X,Y) \to F(X,Y)$ over $R$ such that $f(g(X))=X=g(f(X))$. 
\end{itemize}
\end{defn}
\begin{rem}
Note that a necessary and sufficient condition for $f$ to be an isomorphism is that its Jacobian matrix $J(f):=\left[\frac{\partial}{\partial x_j}f_i(0) \right]_{i,j=1}^{n}$ is invertible (see \cite[Section 9.4]{MH}).
\end{rem}
\begin{exmp} \label{defmultiplicationmaps}
For any integer $k \in \Z$ and any formal group $F(X,Y)$ over $R$, one defines recursively the multiplication-by-$k$ homomorphism $[k]_{F} : F(X,Y) \to F(X,Y)$ as follows : 
\[\displaystyle \left\{\begin{array}{l} 
[0]_{F}(X) = \vec{0} \ ; \ [1]_{F}(X) = X ; \\
\text{if $k \geq 2$, then } [k]_{F}(X) := X +_{F} [k-1]_{F}(X) = F(X,[k-1]_{F}(X)) \\
\text{if $k <0$, then } [k]_{F}(X) := \iota([-k]_{F}(X)) \ . 
\end{array}\right.\]
\end{exmp}
\begin{defn}\label{defstrictiso}
An isomorphism $f : F(X,Y) \to G(X,Y)$ is said to be {\it strict} when its Jacobian matrix $J(f)$ is the identity matrix. This amounts to say that $f(X) \equiv X$ mod degree 2 terms.
\end{defn}
The next proposition provides what we need to define the logarithm of a formal group. Its proof for $1$-dimensional formal groups is provided by \cite[Section 5.4]{MH}. For higher-dimensional formal groups, it goes virtually the same, but for the sake of completeness, we give in Appendix \ref{sb} a full proof for $2$-dimensional formal groups over $\Z_{p}$, which is the case we are interested in.
	\begin{prop} \label{p1.1} \noindent \\
	Let $R$ be a $\Z_{(p)}$-algebra of characteristic $0$ that is $p$-adically separated and $n$ be a positive integer.
	Let $F(X,Y)$ be an $n$-dimensional formal group over $R$. For any integer $k \geq 0$, set 
	\[ \displaystyle L_{k}(X):= p^{-k}[p^{k}]_{F}(X) \ , \] 
	where $[p]_F$ is a multiplication by $p$ endomorphism of $F$. Then $\displaystyle \lim_{k \to \infty} L_{k}(X)$ converges (in $R \otimes \Q[[X]]$) to a power series $L(X)$ over $R$ that satisfies $J(L) = \mathrm{I}_{n}$ and $L(X +_{F} Y) = L(X) + L(Y)$.\\
	In particular, $L$ defines a strict isomorphism from $F(X,Y)$ to the $n$-dimensional additive formal group $G^{(n)}_{a}$.
	\end{prop}

We end up this preliminary section by introducing an arithmetic invariant of formal group laws, following \cite[Section 18.3]{MH}	 for the positive characteristic case, and extending it to local rings of characteristic $0$ (e.g. finite extensions of $\Z_{p}$) by reduction modulo $p$ (see \cite[(18.3.10)]{MH}).
	 \begin{defn} \label{d1.3} 
	 Assume that $\FF$ is a perfect field of characteristic $p$, that $n$ is a positive integer and that $F(X,Y)$ is an $n$-dimensional formal group over $\FF$. Write $[p]_{F}(X) = (H_{1}(X), \ldots , H_{n}(X))$. Then $F$ is said to be {\it of finite height}  if the ring $\FF[[x_{1}, \dots , x_{n}]]$ is a finitely generated free module over the subring $\FF[[H_{1}(X), \ldots , H_{n}(X)]]$. In this case, the rank of this module is of the form $p^{h}$ for some integer $h \geq 0$ : the latter is called the {\it height of $F(X,Y)$}. 
	 \end{defn}
	 
	 \begin{defn} \label{d1.3bis}
	 Let $R$ be a local ring of characteristic $0$ and residue field $\FF$ of characteristic $p$.
	 Let $F(X,Y)$ be an $n$-dimensional formal group over $R$. We define the {\it height of $F(X,Y)$} as the height (as given by Definition \ref{d1.3}) of its reduction over $\FF$ (i.e. modulo the maximal ideal of $R$).
	  \end{defn}
	  \begin{exmp} \noindent 
\begin{itemize}
\item[\tt $(i)$] The multiplicative group $\mathbb{G}_{m}$ has height $1$.
\item[\tt $(ii)$] For any integer $n \geq 1$, the $n$-dimensional additive formal group $\mathbb{G}_{a}^{(n)}$ is of infinite height.
\item[\tt $(iii)$] Let us consider the $2$-dimensional formal group defined over $\Z_{p}$ (as in Example \ref{exFG}) by 
\[F(X,Y) = \displaystyle \left( x_{1} + y_{1} + x_{2}y_{2}; \ x_{2} + y_{2} \right) \ , \]
and assume $p \geq 3$. An immediate recursion shows that for any integer $n \geq 2$, we have 
\[\displaystyle [n]_{F}(X) = \left(nx_{1} + \frac{(n-1)(n-2)}{2}x_{2}^{2} \ ; \ nx_{2} \right) \ .  \]
For $n = p$, the corresponding formal group over $\FF_{p}$ has its multiplication by $p$ endomorphism given by : 
\[\displaystyle [p]_{F}(X) \text{ mod } p = \left( \alpha x_{2}^{2} \ ; \ 0 \right) \ , \]
where $\alpha$ denotes the reduction modulo $p$ of $\frac{(n-1)(n-2)}{2}$. This proves that this formal group is not of finite height, as the $\FF_{p}[[x_{2}^{2}]]$-module $\FF_{p}[[x_{1}, x_{2}]]$ is not of finite rank.
\end{itemize}
	\end{exmp}
	
	\section{Constructing 2-dimensional Lubin-Tate formal group laws over $\Z_{p}$} \label{s2}
	
	In this section, we extend the 1-dimensional Lubin-Tate formal groups over $\Z_{p}$ as originally defined by Lubin and Tate to the following $2$-dimensional analogue. Unless otherwise stated, we set $X = (x_{1}, x_{2})$ and $Y = (y_{1}, y _{2})$ in this section.
	
	Given two positive integers $k, \ell$ that are relatively prime, we first set 
	\begin{equation}
	\label{defCLThl}
	\displaystyle \CCC_{LT}^{k,\ell} := \{f(X) \in \Z_{p}[[X]]^{2} \ \vert \ f(X) \equiv (px_{1}, px_{2}) \text{ mod degree 2 and } f(X) \text{ mod } p \in \{ (x_{1}^{p^{k}}, x_{2}^{p^{\ell}}) \ ; \ (x_{2}^{p^{k}}, x_{1}^{p^{\ell}})) \} \ .
	\end{equation}
	We then define $\CCC_{LT}$ as the union of all $\CCC_{LT}^{k, \ell}$ for $(k,\ell)$ running over pairs of relatively prime positive integers : 
	\begin{equation}
	\label{defCLTdim2}
	\CCC_{LT} := \displaystyle \bigcup_{k >0, \ \ell >0 \ \vert \ \gcd(k,\ell) = 1} \CCC_{LT}^{k,\ell} \ .
	\end{equation}
	Let us already point out that, unlike the $1$-dimensional case, not all elements of $\CCC_{LT}$ will be associated with a ($2$-dimensional) Lubin-Tate formal group: we provide in Remark \ref{r1} below an explicit counter-example in this direction. 
	
	The key to construct a $2$-dimensional Lubin-Tate formal group with multiplication by $p$ given by a prescribed element $f$ of $\CCC_{LT}$ is to define this formal group via its logarithm, as introduced in Proposition \ref{p1.1}: we then obtain the formal group we want by exponentiation of this logarithm. Note that the properties of this logarithm, hence of the $2$-dimensional formal group build from it, will later impose some necessary conditions on our choice of $f$ to ensure that it defines an endomorphism of the formal group. This makes a huge difference with the $1$-dimensional case, where any $f$ was eligible (see \cite{JL2}).
	\subsection{The functional equation lemma}
Following \cite[Section 10]{MH}, we recall some notation and results required to build higher-dimensional formal groups as claimed above. 

Let $R$ be a discrete valuation ring with residue field $\FF$ of characteristic $p$ and let $K = \mathrm{Frac}(R)$ be the fraction field of $R$. Let $I$ be an ideal of $R$ and let $\sigma : K \to K$ be a field endomorphism that behaves on $R$ like a Frobenius morphism modulo $I$, i.e. such that there exists a power $q$ of $p$ satisfying:
\[\displaystyle \forall \ x \in R, \sigma(x) \equiv x^{q} \mod I \ .\]
Let $(s_{i})_{i \in \Z_{+}^{*}}$ be a family of  $n \times n$ matrices with coefficients in $K$ such that : 
\begin{equation}
\label{condmatricesi}
\forall \ i \geq 1, \ s_{i}I \subset R \ .
\end{equation}
Let $n$ and $m$ be positive integers. Given an $n$-tuple of power series $g(X) = g(x_{1}, \ldots, x_{m})$ with coefficients in $R$ and without constant terms (i.e. having all its components in $XR[[X]]$), we construct an $n$-tuple of power series $L_{g}(X)$ by means of the following recursion formula (also called {\it functional equation}) : 
	 	\begin{equation} \label{e2}
	 		\displaystyle L_{g}(X) =g(X)+\sum_{i \geq 1} s_{i} \sigma_{*}^{i}L_{g}(X^{q^{i}}) \ ,
	 	\end{equation}
where $\sigma_{*}^{i}L_{g}(X)$ is obtained from $L_{g}(X)$ by applying the endomorphism $\sigma^{i}$ to the coefficients of the $n$-tuple of power series $L_{g}(X)$, and where we set for short $X^{q^{i}} := (x_{1}^{q^{i}}, \ldots , x_{m}^{q^{i}})$.
\begin{defn} \label{deftypeofL}
With the notation above, if $g$ satisfies the recursion relation \eqref{e2}, we say that $L_{g}(X)$ satisfies {\it a functional equation of type $(s_{i})_{i \in \Z_{+}^{*}}$.}
\end{defn}
The next statement, known as the {\it $n$-dimensional functional equation lemma} \cite[(10.2)]{MH}, will be one of our key tools to build $2$-dimensional Lubin-Tate formal groups.
	 \begin{lem}\label{l2.1} 
	 	Keep the same notation as above and define $L_{g}(X)$ by the recursion formula \eqref{e2}. Also assume that the Jacobian matrix $J(L_{g})$ of $L_{g}$ is invertible. Then the following hold.
		\begin{enumerate} 
		 \item There is a unique $n$-tuple of power series $L_{g}^{-1}(X)$ such that 
		 \[ L_{g}(L_{g}^{-1}(X)) = X =L_{g}^{-1}(L_{g}(X)) \ . \]
		 \item The $n$-tuple of power series $F_{g}(X,Y) := L_{g}^{-1}(L_{g}(X) + L_{g}(Y))$ has coefficients in $R$. 
		 \item For any $n$-tuple of power series $h(Y) \in (YR[[Y]])^{n}$ in $m$ indeterminates, the $n$-tuple of power series $L_{g}^{-1}(L_{h}(Y))$ has coefficients in $R$, and $L_{g}(h(Y))$ satisfies a functional equation of type $(s_{i})_{i \in \Z_{+}^{*}}$.
		 \item Assume that $\gamma(X)$ is an $n$-tuple of power series with coefficients in $R$ and that $\delta(X)$ is an $n$-tuple of power series with coefficients in $K$ (both with $m$ indeterminates). Then : $\forall \ r \geq 1$,
		 \[\displaystyle \left( \gamma(X) \equiv \delta(X) \ \mathrm{mod} \ I^{r} \right) \iff \left(L_{g}(\gamma(X)) \equiv L_{g}(\delta(X)) \ \mathrm{mod} \ I^{r} \right) \ . \]
		 \end{enumerate}
	 \end{lem}
	 
\subsection{Application to $2$-dimensional Lubin-Tate formal groups over $\Z_{p}$} Using the functional equation lemma, we now construct a $2$-dimensional analogue of the logarithm for $1$-dimensional Lubin-Tate formal groups over $\Z_{p}$, from which we will deduce an appropriate $2$-dimensional analogue of Lubin-Tate formal groups over $\Z_{p}$. 
	 
Let $R := \Z_{p}$, so that $K = \Q_{p}$ is the field of $p$-adic numbers and $\FF = \FF_{p}$ is the field with $p$ elements. Let $I = p\Z_{p}$ be the maximal ideal of $R$ and $q = p$ be the cardinality of the residue field $\FF$: we can then take $\sigma = \mathrm{Id}$. Let $h_{1}$ and $h_{2}$ be fixed relatively prime positive integers, and consider the family $(s_{i})_{i \in \Z_{+}^{*}}$ of $2 \times 2$ matrices with coefficients in $p^{-1}\Z_{p} \subset K$ defined as follows: 
\[\displaystyle s_{h_{1}} := \left(\begin{array}{cl}0 & p^{-1} \\ 0 & 0 \end{array}\right) \ , \ s_{h_{2}} := \left(\begin{array}{cc}0 & 0 \\ p^{-1} & 0  \end{array}\right) \ \text{ and } \ s_{i} = 0 \text{ if } i \not\in \{h_{1}, h_{2}\} \ . \]
Finally, let $g(X) = X = (x_{1}, x_{2}) \in (XR[[X]])^{2}$. Then we can define $L_{g}$ by the recursive formula \eqref{e2}. This candidate for the logarithm of our $2$-dimensional Lubin-Tate formal group satisfies the following nice relation. 
	\begin{prop} \label{p2.2}
With the notation above, the logarithm $L_{g}(X)=(L_{1}(x_{1}, x_{2}), L_{2}(x_{1}, x_{2}))$ satisfies the following recursion formulas:  
		\begin{equation} \label{e3}
		\left\{ \begin{array}{l}
		L_{1}(x_{1}, x_{2}) = x_{1} + p^{-1}L_{2}(x_{1}^{p^{h_{1}}}, x_{2}^{p^{h_{1}}}) \\
		L_{2}(x_{1}, x_{2}) = x_{2} + p^{-1}L_{1}(x_{1}^{p^{h_{2}}}, x_{2}^{p^{h_{2}}}) \ .
		\end{array}
		\right.
		\end{equation}
In particular, this provides the following explicit formulas for the power series composing $L_{g}(X)$: 
		\begin{eqnarray}   
   \label{e4} \displaystyle L_{1}(x_{1}, x_{2}) = x_{1} + \sum_{k \geq 1} p^{-2k} x_{1}^{p^{k(h_{1} + h_{2})}} + \sum_{k \geq 0} p^{-(2k+1)} x_{2}^{p^{h_{1} + k(h_{1} + h_{2})}} \ ; \\
   \label{e5} \displaystyle L_{2}(x_{1}, x_{2}) = x_{2} + \sum_{k \geq 1} p^{-2k} x_{2}^{p^{k(h_{1} + h_{2})}} + \sum_{k \geq 0} p^{-(2k+1)} x_{1}^{p^{h_{2} + k(h_{1} + h_{2})}}  \ .
		\end{eqnarray}
	\end{prop}
	\begin{proof}
	Applying the recursion formula \eqref{e2} in our setting directly provides \eqref{e3}, since the only two non-zero elements in the family $(s_{i})_{i \in \Z_{*}^{+}}$ are $s_{h_{1}} := \left(\begin{array}{cl}0 & p^{-1} \\ 0 & 0 \end{array}\right)$ and $s_{h_{2}} := \left(\begin{array}{cc}0 & 0 \\ p^{-1} & 0  \end{array}\right)$: 
		\begin{align*}
			L_{g}(X)=\begin{pmatrix}
				L_{1}(x_{1}, x_{2}) \\
				L_{2}(x_{1}, x_{2})
			\end{pmatrix}&=\begin{pmatrix}
			x_{1} \\
			x_{2}
		\end{pmatrix}+\begin{pmatrix}
		0 & \frac{1}{p} \\ 0 & 0
	\end{pmatrix} \begin{pmatrix}
	L_{1}( x_{1}^{p^{h_{1}}}, x_{2}^{p^{h_{1}}})\\
	L_{2}(x_{1}^{p^{h_{1}}}, x_{2}^{p^{h_{1}}})
\end{pmatrix}+\begin{pmatrix}
0 & 0 \\ \frac{1}{p} & 0
\end{pmatrix}  \begin{pmatrix}
	L_{1}( x_{1}^{p^{h_{2}}}, x_{2}^{p^{h_{2}}})\\
	L_{2}(x_{1}^{p^{h_{2}}}, x_{2}^{p^{h_{2}}})
\end{pmatrix} \\
&=\begin{pmatrix} x_{1}+\frac{1}{p}L_{2}(x_{1}^{p^{h_{1}}},x_{2}^{p^{h_{1}}})\\ x_{2}+ \frac{1}{p}L_{1}(x_{1}^{p^{h_{2}}},x_{2}^{p^{h_{2}}}) \end{pmatrix} \\
& = X + \frac{1}{p} \begin{pmatrix} L_{2}(x_{1}^{p^{h_{1}}},x_{2}^{p^{h_{1}}}) \\ L_{1}(x_{1}^{p^{h_{2}}},x_{2}^{p^{h_{2}}})\end{pmatrix} \ . 
		\end{align*}
We use the same kind of recursion to obtain the explicit formulae \eqref{e4} and \eqref{e5} by induction. Indeed, we directly have that :
\begin{align*}
L_{g}(X) & = X + s_{h_{1}}L_{g}(X^{p^{h_{1}}}) + s_{h_{2}}L_{g}(X^{p^{h_{2}}})\\
& = X + s_{h_{1}}\left(X^{p^{h_{1}}} + s_{h_{1}}L_{g}(X^{p^{2h_{1}}}) + s_{h_{2}} L_{g}(X^{p^{h_{1} + h_{2}}}) \right) + s_{h_{2}}\left(X^{p^{h_{2}}} + s_{h_{1}}L_{g}(X^{p^{h_{2} + h_{1}}}) + s_{h_{2}}L_{g}(X^{p^{2h_{2}}}) \right) \\
& = X + s_{h_{1}}X^{p^{h_{1}}} + s_{h_{2}}X^{p^{h_{2}}} + (s_{h_{1}}s_{h_{2}} + s_{h_{2}}s_{h_{1}})L_{g}(X^{p^{h_{1} + h_{2}}}) + s_{h_{1}}^{2} L_{g}(X^{p^{2h_{1}}}) + s_{h_{2}}^{2} L_{g}(X^{p^{2h_{2}}}) \\
& = X + \frac{1}{p} \begin {pmatrix} x_{2}^{p^{h_{1}}} \\ 0 \end{pmatrix} + \frac{1}{p} \begin{pmatrix} 0 \\ x_{1}^{p^{h_{2}}}\end{pmatrix} + (s_{h_{1}}s_{h_{2}} + s_{h_{2}}s_{h_{1}})L_{g}(X^{p^{h_{1} + h_{2}}}) \\
& = \begin{pmatrix}
x_{1} + p^{-1}x_{2}^{p^{h_{1}}} \\
x_{2} + p^{-1}x_{1}^{p^{h_{2}}}
\end{pmatrix} + p^{-2}L_{g}(X^{p^{h_{1} + h_{2}}}) \ , 
\end{align*}
since $s_{h_{1}}$ and $s_{h_{2}}$ are both nilpotent of order $2$ while $s_{h_{1}}s_{h_{2}} + s_{h_{2}}s_{h_{1}} = p^{-2}\mathrm{Id}$. Let us go one step further to see how the induction process actually works. Replacing $X$ by $X^{p^{h_{1} + h_{2}}}$ in the previous calculations provides that
\[\displaystyle L_{g}(X^{p^{h_{1} + h_{2}}}) = X^{p^{h_{1} + h_{2}}} + \frac{1}{p}\begin{pmatrix} x_{2}^{p^{2h_{1} + h_{2}}} \\ 0 \end{pmatrix} + \frac{1}{p}\begin{pmatrix} 0 \\ x_{1}^{p^{h_{1} + 2h_{2}}} \end{pmatrix} + p^{-2}L_{g}(X^{p^{2(h_{1} + h_{2})}}) \ , \]
hence that 
\[\displaystyle L_{g}(X) = \begin{pmatrix}
x_{1} + p^{-1}x_{2}^{p^{h_{1}}} \\
x_{2} + p^{-1}x_{1}^{p^{h_{2}}}
\end{pmatrix} + p^{-2}\left[\begin{pmatrix} x_{1}^{p^{h_{1}+h_{2}}} + p^{-1}x_{2}^{p^{2h_{1} + h_{2}}} \\ x_{2}^{p^{h_{1} + h_{2}}} + p^{-1} x_{1}^{p^{h_{1} + 2h_{2}}} \end{pmatrix} + p^{-2}L_{g}(X^{p^{2(h_{1} + h_{2})}})\right] \ .\]
This can be rewritten as 
\begin{equation}
\label{induceq} L_{g}(X) = \begin{pmatrix} x_{1} + p^{-1}x_{2}^{p^{h_{1}}} + p^{-2}x_{1}^{p^{h_{1}+h_{2}}} + p^{-3}x_{2}^{p^{2h_{1} + h_{2}}} \\
x_{2} + p^{-1}x_{1}^{p^{h_{2}}} + p^{-2}x_{2}^{p^{h_{1} + h_{2}}} +p^{-3}x_{1}^{p^{h_{1} + 2h_{2}}}  \end{pmatrix} 
+ p^{-4} L_{g}(X^{p^{2(h_{1} + h_{2})}}) \ , 
\end{equation}
which shows exactly the first terms appearing in \eqref{e4} and \eqref{e5} -- for $k = 0$ and $k = 1$. The shape of the residue in \eqref{induceq} ensures that a direct induction on $n \geq 2 $ (to transform the residue $p^{-2n}L_{g}(X^{p^{n(h_{1} + h_{2})}})$ into $p^{-2(n+1)}L_{g}(X^{p^{(n+1)(h_{1} + h_{2})}})$) provides the infinite sum formulae \eqref{e4} and \eqref{e5} of the statement and hence, it finishes the proof.
\end{proof}
This proposition ensures in particular that $J(L_{g}) = \mathrm{I}_{2}$ is invertible; hence, we can apply the $2$-dimensional functional equation (Lemma \ref{l2.1}) to $L_{g}$ and set the following definition.
\begin{defn}
\label{defLTFGdim2}
We define the {\it $2$-dimensional Lubin-Tate formal group $F(X,Y)$ over $\Z_{p}$} by
\[\displaystyle F(X,Y) := L^{-1}(L(X) + L(Y)) \ , \]
where $L = L_{g}$ is the logarithm formal series provided by Proposition \ref{p2.2}.
\end{defn}
\begin{rem}
The formal group introduced in Definition \ref{defLTFGdim2} depends on the pair of relatively prime integers $\vec{h} = (h_{1}, h_{2})$, hence we actually defined a family of $2$-dimensional Lubin-Tate formal groups over $\Z_{p}$, as announced.
\end{rem}

\subsection{On endomorphisms of these formal groups} As above, we fix a pair $\vec{h} = (h_{1}, h_{2})$ of positive integers that are relatively prime and we denote by $F(X,Y)$ the $2$-dimensional Lubin-Tate formal group over $\Z_{p}$ associated with $\vec{h}$ by Definition \ref{defLTFGdim2}. The next proposition imposes explicit conditions on the multiplication-by-$p$ endomorphism of $F(X,Y)$.
\begin{prop} \label{congruencesdim2}
The endomorphism $[p]_{F}(X)$ must satisfy
\begin{equation} \label{e6}
[p]_{F}(x_{1}, x_{2}) \equiv (px_{1}, px_{2}) \mathrm{ \ mod \ degree \ 2 \ and \ } [p]_{F}(x_{1},x_{2}) \equiv (x_{2}^{p^{h_{1}}}, x_{1}^{p^{h_{2}}}) \mathrm{ \ mod\  }p \ . 
\end{equation}
In particular, it is an element of $\CCC^{h_{1}, h_{2}}_{LT} \subset \CCC_{LT}$ (see \eqref{defCLThl}).
\end{prop}
\begin{proof}
Recall that, by definition of multiplication-by-$p$ in a formal group (see Example \ref{defmultiplicationmaps}), we have $[p]_{F}(X) = F(X, [p-1]_{F}(X))$. The formula provided by Definition \ref{defLTFGdim2} then ensures that
\[\displaystyle  \begin{array}{rcl} [p]_{F}(X) &=& L^{-1}(L(X) + L([p-1]_{F}(X))) \\ 
& = & L^{-1}(L(X) + L(F(X,[p-2]_{F}(X)))) \\
& = & L^{-1}(L(X) + L(X) + L([p-2]_{F}(X)) \\
& = & L^{-1}(2L(X) + L([p-2]_{F}(X))) \ .\end{array}\]
Since $p \geq 2$ and $[0]_{F}(X) = (0,0)$, we obtain recursively that 
\begin{equation}
\label{formulemultp}
[p]_{F}(X) = L^{-1}(pL(X)) \ ,
\end{equation}
which already proves (by Lemma \ref{l2.1}) that the tuple of power series defining the endomorphism $[p]_{F}$ has coefficients in $\Z_{p}$, while Proposition \ref{p1.1} certifies that it is an endomorphism of $F(X,Y)$. Now, note that we can rewrite \eqref{formulemultp} as 
\begin{equation} \label{linearity}
\displaystyle L([p]_{F}(X)) = pL(X) \ .
\end{equation}
Using Proposition \ref{p2.2}, we obtain that 
\[\displaystyle \begin{array}{rcl}
L([p]_{F}(X)) & = & pX + \begin{pmatrix}
x_{2}^{p^{h_{1}}} + \displaystyle \sum_{k \geq 1} p^{-2k} x_{2}^{p^{h_{1} + k(h_{1} + h_{2})}} + \sum_{k \geq 0} p^{-(2k+1)} x_{1}^{p^{(k+1)(h_{1}+h_{2})}}  \\
x_{1}^{p^{h_{2}}} + \displaystyle \sum_{k \geq 1} p^{-2k} x_{1}^{p^{h_{2} + k(h_{1} + h_{2})}} + \sum_{k \geq 0} p^{-(2k+1)} x_{2}^{p^{(k+1)(h_{2} + h_{1})}}
\end{pmatrix} \\
& & \\
& = & pX + L(\Phi(X)) 
\end{array}\]
with $\Phi(X) = (x_{2}^{p^{h_{1}}}, x_{1}^{p^{h_{2}}}) \in X^{2}\Z_{p}[[X]]$. On the one hand, this directly shows that 
\begin{equation}
\label{e7}
\displaystyle L([p]_{F}(X)) \equiv L(\Phi(X)) \mathrm{\ mod \ } p \ , 
\end{equation}
thus the last statement of Lemma \ref{l2.1} implies that $[p]_{F}(X) \equiv \Phi(X) \mathrm{ \ mod \ } p$, which is the mod $p$ congruence we claimed. On the other hand, equality \eqref{linearity} obviously implies that 
\begin{equation}
\displaystyle L([p]_{F}(X)) \equiv pL(X) \mathrm{\ mod \ degree \ 2 \ terms}  \ . 
\end{equation}
The formulae for the logarithm given by Proposition \ref{p2.2} allow us to rewrite this congruence as 
\begin{equation}
\displaystyle [p]_{F}(X) \equiv pX \mathrm{\ mod \ degree \ 2 \ terms}  \ ,
\end{equation}
which is exactly what we claimed, and the proof is now complete.
\end{proof}

\begin{rem}
From the modulo $p$ congruence provided by Proposition \ref{congruencesdim2}, we see that $\FF_{p}[[x_{1}, x_{2}]]$ is free of finite rank $p^{h_{1} + h_{2}}$ as an $\FF_{p}[[x_{1}^{p^{h_{2}}}, x_{2}^{p^{h_{1}}}]]$-module, which means that the $2$-dimensional Lubin-Tate formal group $F$ we built is of finite height $h_{1} + h_{2}$ (see Definition \ref{d1.3}).
\end{rem}

The way we defined multiplication-by-$p$ for the $2$-dimensional Lubin-Tate formal group $F$ introduced in Definition \ref{defLTFGdim2} can be used to define, more generally, multiplication by any element of $\Z_{p}$ in $F$ as follows.
\begin{prop} \label{c2.3.1}
Let $F$ be the $2$-dimensional Lubin-Tate formal group given by Definition \ref{defLTFGdim2} and $a$ be any element of $\Z_{p}$. Then the formula 
\begin{equation}
\label{endomultarbitrary}
[a]_{F}(X) := L^{-1}(aL(X))
\end{equation}
defines an endomorphism of $F$ (with coefficients in $\Z_{p}$).
\end{prop}
\begin{proof}
First note that if $a$ is an integer, then the first part of the proof of Proposition \ref{congruencesdim2} with $[p]_{F}$ replaced by $[a]_{F}$ works the same, hence \eqref{e6} is satisfied by $[a]_{F}$ instead of $[p]_{F}$ and the same argument as for $[p]_{F}$ shows that $[a]_{F}$ is an endomorphism of $F$ with coefficients in $\Z_{p}$. \\
For arbitrary $a$ in $\Z_{p}$, one can either use a continuity argument to deduce the result from the previous case, or notice that multiplying by $a$ the recursion formula of Proposition \ref{p2.2} shows that $aL(X)$ satisfies the same functional equation as $L(X)$: this allows us to use Lemma \ref{l2.1} and Proposition \ref{p1.1} (as we did for $[p]_{F}$ in the proof of Proposition \ref{congruencesdim2}) to conclude.
\end{proof}

\begin{rem}\label{iteratedmultp}
Let us point out here that a direct induction from \eqref{linearity} ensures that
\begin{equation}\label{eqiteratedmultp}
\forall \ n \in \Z_{\geq 1}, \ \displaystyle [p^{n}]_{F}(X) = L^{-1}(p^{n}L(X)) \ .
\end{equation}
\end{rem}

\begin{rem} \label{r1}
As we claimed earlier, not every element of $\CCC_{LT}$ is a priori a good candidate to provide an endomorphism of a $2$-dimensional Lubin-Tate formal group. Indeed, assume that $h_{1}, h_{2}$ are both greater than or equal to $2$ and consider the element $f_{\vec{h}} \in \CCC^{h_{1}, h_{2}}_{LT}$ defined by 
\begin{equation} \label{cexendo}
f_{\vec{h}}(x_{1}, x_{2}) = \begin{pmatrix} px_{1} + x_{1}^{p^{h_{1}}} \\ px_{2} + x_{2}^{p^{h_{2}}} \end{pmatrix} \ . 
\end{equation}
One directly checks that it cannot provide a multiplication-by-$p$ endormorphism since it does not satisfy the same relation as \eqref{e7} with $[p]_{F}$ replaced by $f_{\vec{h}}$. Note that the same kind of argument (based on Proposition \ref{c2.3.1}) proves that $f_{\vec{h}}$ cannot provide a multiplication-by-$a$ endomorphism for any $a \in \Z$, and the density of $\Z$ in $\Z_{p}$ then ensures that $f_{\vec{h}}$ cannot provide any endomorphism of $F$ defined by multiplication by a scalar in $\Z_{p}$.

\noindent Using the logarithm recursion formulae proven in Proposition \ref{p2.2} and doing a little more calculation, one can actually prove that $f_{\vec{h}}$ cannot provide any endomorphism of $F$ over $\Z_{p}$.
\end{rem}
		
	\section{Torsion points and Newton copolygon of $2$-dimensional formal groups} \label{s3}
		Let $\vec{h} = (h_{1}, h_{2})$ be a pair of positive coprime integers and let $F$ be the $2$-dimensional Lubin-Tate formal group associated with $\vec{h}$ by Definition \ref{defLTFGdim2}. Recall that in this setting, the formal series $L$ is called the {\it logarithm of the formal group $F$}.
		
		The goal of this section is to define the $p^{n}$-torsion part of $F$ for any integer $n \geq 1$, to give some basic properties of it (a deeper study will be initiated in the next section), then to provide a geometric point of view on all these objects by introducing the Newton copolygon associated with $F$ and studying some explicit examples.
		
	\subsection{The $p^{\infty}$-torsion of a $2$-dimensional Lubin-Tate formal group} \label{s31}
	For any integer $n \geq 2$, we define the {\it $p^{n}$-torsion points of $F$} by 
	\begin{equation} \label{eqdeftorsionpoints}
	\displaystyle F[p^{n}]=\{X\in \QQ_{p}^{2} \ \vert \ [p^{n}]_{F}(X)= \vec{0} \ \} \ . 
	\end{equation}
Since the solutions of $[p^{n}]_{F}(X) = \vec{0}$ define an affine $2$-space over $\Q_{p}$, the set of all coordinates of the points of $F[p^{n}]$ generates a field extension of $\Q_{p}$ that will be denoted by $\Q_{p}(F[p^{n}])$.	

			 \begin{lem} \label{l4.1}
For any integer $n \geq 2$, $F[p^{n}]$ is a $\mathbb{Z}_{p}$-module.
	 \end{lem}

\begin{proof}
Let $n \geq 2$ be an integer and let $\mathrm{End}_{\Z_{p}}(F)$ denote the ring of endomorphisms of $F$ with coefficients in $\Z_{p}$. By Proposition \ref{c2.3.1}, we have a map $\mu : \left[ a \in \Z_{p} \to [a]_{F} \in \mathrm{End}_{\Z_{p}}(F) \right]$ that actually happens (by \eqref{endomultarbitrary}) to be a ring homomorphism. One can then check that $F[p^{n}]$ is actually a $\Z_{p}$-module for the addition law defined by $F$ (i.e. $\alpha + \beta := F(\alpha, \beta)$) and for the scalar multiplication defined by $\mu$ (i.e. $a \cdot \alpha := \mu(a)(\alpha) = [a]_{F}(\alpha)$).
\end{proof}

One of the most interesting features of these torsion points is given by the following statement.
	  \begin{thm} \label{t3.3}
	  For any integer $n \geq 1$, the solutions of $[p^{n}]_{F}(X) = \vec{0}$ are all multiplicity one.  
	  \end{thm}
	  
	  \begin{proof}
	  Thanks to Remark \ref{iteratedmultp}, we are reduced to prove the result for $n = 1$. Let $\alpha = (\alpha_{1}, \alpha_{2})$ be such that $[p]_{F}(\alpha) = \vec{0}$. Since $[p]_{F}$ is an endomorphism of $F$, we have 
	  \begin{equation}
	  \label{e11}
	  [p]_{F}(X + \alpha) = [p]_{F}(X) + [p]_{F}(\alpha) = [p]_{F}(X) \ . 
	  \end{equation}
Let us write $\delta[p]_{F}$ for the differential map associated with $[p]_{F}$ and $F(X,Y) = (F_{1}(X,Y), F_{2}(X,Y))$, and recall that $X + \alpha = F(X, \alpha)$ by definition. Then, differenciating \eqref{e11} with respect to $X$ shows that 
\[\displaystyle \delta[p]_{F}(X + \alpha)\partial_{X}F(X,\alpha) = \delta[p]_{F}(X) \ . \]
Evaluating this equality at $X = \vec{0}$ ensures in particular that 
\begin{equation}
\label{e11bis}
\delta[p]_{F}(\alpha)\partial_{X}F(\vec{0}, \alpha) = \delta[p]_{F}(\vec{0}) \ , \text{ i.e. } \delta[p]_{F}(\alpha)\partial_{X}F(\vec{0}, \alpha) = p\mathrm{I}_{2} \ .
\end{equation} 
(The last equality follows from \eqref{formulemultp} and $J(L) = \mathrm{I}_{2}$.)

Now, recall that by definition of formal groups, $F(X,Y)$ satisfies
\[\displaystyle F(X,Y) \equiv X + Y \textrm{ mod degree } 2 \textrm{ terms; } \]
hence, we have $F_{i}(X,Y) \equiv x_{i} + y_{i} \textrm{ mod degree $2$ terms}$ for any $i \in \{1,2\}$. This implies that 
\[\displaystyle \left( \frac{\partial F_{j}}{\partial x_{i}}\right)_{1 \leq i,j \leq 2}  = \mathrm{I}_{2} \ , \text{ i.e. } \partial_{X}F \equiv \mathrm{Id} \ . \]
Replacing this in \eqref{e11bis}, we conclude that $\delta[p]_{F}(\alpha) = p \mathrm{I}_{2}$ is invertible, hence $\alpha$ is a multiplicity-free solution of $[p]_{F}(X) = \vec{0}$, as claimed.
	  \end{proof}

The next proposition shows that the $p^{n}$-torsion part of a $2$-dimensional Lubin-Tate formal group only depends on the functional equation satisfied by the logarithm of this formal group. 
	\begin{prop} \label{p4.2}
	 Let $G$ be a $2$-dimensional Lubin-Tate formal group (in the sense of Definition \ref{defLTFGdim2}) whose logarithm satisfies a functional equation of the same type $(s_{i})_{i \geq 1}$ as $L$ (in the sense of Definition \ref{deftypeofL}). Then, for any integer $n \geq 2$, we have $F[p^{n}] \cong G[p^{n}]$ as $\Z_{p}$-modules.
	 \end{prop}
	\begin{proof}
	Let $L_{\gamma}$ denote the logarithm of the formal group $G$, with $\gamma$ being an $n$-tuple of power series with coefficients in $\Z_{p}$ and no constant terms. We then have, by definition, 
	\[ \displaystyle G(X,Y) = L_{\gamma}^{-1}(L_{\gamma}(X) + L_{\gamma}(Y)) \ . \]
	Our assumption means that the functional equation \eqref{e2} satisfied by $L_{\gamma}$ involves the same family $(s_{i})_{i \geq 1}$ of matrices than the functional equation of $L$. The functional equation lemma (Lemma \ref{l2.1}) then ensures that $L_{\gamma}^{-1}$ is well-defined and that $L_{\gamma}^{-1} \circ L$ has coefficients in $\Z_{p}$. It is now straightforward to check that $L_{\gamma}^{-1} \circ L$ actually defines an isomorphism of formal groups from $F$ to $G$, which implies in particular that their $p^{n}$-torsion points provide isomorphic $\Z_{p}$-modules for any integer $n \geq 1$.	
	\end{proof}	

Thanks to Proposition \ref{p4.2}, we can choose any convenient $\gamma$ such that $L_{\gamma}$ satisfies the same functional equation as $L$ to study the $p^{\infty}$-torsion part of $F$ via its counterpart for the $2$-dimensional Lubin-Tate formal group associated with $L_{\gamma}$. Considering the formulae given by Proposition \ref{congruencesdim2} for $[p]_{F}$, it is natural to introduce the following dynamical system $\DDD_{\vec{h}}$ as a good approximation of $[p]_{F}$: 
\begin{equation}
\label{defDDD}
\displaystyle \DDD_{\vec{h}}\begin{pmatrix} x_{1} \\ x_{2} \end{pmatrix} := \begin{pmatrix} px_{1} + x_{2}^{p^{h_{1}}} \\ px_{2} + x_{1}^{p^{h_{2}}}\end{pmatrix} \ .
\end{equation}
Note that $\DDD_{\vec{h}}$ does satisfy the same conditions as those given on $[p]_{F}$ by Proposition \ref{congruencesdim2}, but we do not claim for now that it is an endomorphism of the formal group $F$. Nevertheless, this is the simplest approximation of $[p]_{F}$, which avoids further computational complexity, and it is $p$-adically close enough to $[p]_{F}$ so that the qualitative questions we are interested in (such as ramification properties of the field extensions defined by the $p^{\infty}$-torsion part of $F$) can be at least partially answered by a study of the dynamical system $\DDD_{\vec{h}}$. This will lead us in the sequel to the following abuse of notation: when no confusion is possible, we will simply denote $[p]_{F}$ instead of $\DDD_{\vec{h}}$. 

\subsection{Newton copolygons associated with $p$-adic power series in $2$ variables}
The goal of this subsection is to introduce an analogue, for $p$-adic power series in $2$ variables (such as those involved in the endomorphisms of the $2$-dimensional Lubin-Tate formal groups we defined in the previous section), of the Newton copolygon associated with $p$-adic power series in $1$ variable. The latter was introduced by Lubin \cite{JL1} to provide a new point of view on some standard fundamental results in (higher) ramification theory. Our motivation is to build a tool that encodes the same information as usual Newton polygons regarding valuations, but in a more functorial way with respect to the composition of power series, and to provide a more geometrical point of view on the (endomorphisms of the) $2$-dimensional Lubin-Tate formal groups, in the view of studying the ramification properties associated with their $p^{\infty}$-torsion part in the next section (and in forthcoming work).

\subsubsection{Definition of the Newton copolygon associated with a $p$-adic power series in $2$ variables}
Let $f \in \Z_{p}[[x_{1}, x_{2}]]$ be a $p$-adic power series in $2$ variables. Let us write it as 
\[\displaystyle f(x_{1}, x_{2}) = \sum_{i,j \geq 0} c_{ij} x^{i}_{1} x^{j}_{2} \ , \ \textrm{ with } c_{ij} \in \Z_{p} \ . \]
In the $3$-dimensional Euclidean space with coordinates $(\xi_{1}, \xi_{2}, \eta)$, we draw, for each pair of indices $(i,j)$ such that $c_{ij} \not= 0$, the half-space defined by the equation $\eta \leq i \xi_{1}+j \xi_{2}+v(c_{ij})$. The intersection of all these half-spaces provides a convex body that extends downwards. We call the boundary of this convex body the {\it Newton copolygon associated with $f$}, and we denote it by $\NNN_{f}^{c}$. Another way to define the same object is provided by the next definition.
\begin{defn} \label{DefVf}
Let $f \in \Z_{p}[[x_{1}, x_{2}]]$ be a $p$-adic power series in $2$ variables. Let us denote by $(c_{ij})_{i,j \geq 0}$ its coefficients (in $\Z_{p}$) : 
\[\displaystyle f(x_{1}, x_{2}) = \sum_{i, j \geq 0} c_{ij}x_{1}^{i}x_{2}^{j} \ . \]
Set $\AAA$ be the set of all pairs of indices $(i,j)$ such that $c_{ij}$ is non-zero and let $V_{f} : \R^{2} \to \R$ be the function defined by 
\[ \displaystyle \forall \ (\xi_{1}, \xi_{2}) \in \R^{2}, \ V_{f}(\xi_{1}, \xi_{2}) := \min\{ i \xi_{1} + j \xi_{2} + v(c_{ij}), \ (i,j) \in \AAA \} \ . \]
The {\it Newton copolygon associated with $f$} is the graph $\NNN_{f}^{c}$ (in the $3$-dimensional Euclidean space) of the function $V_{f}$. The latter function is called the {\it copolygon function associated with $f$}.
\end{defn}
\begin{rem}
As any simplicial complex living in a $3$-dimensional Euclidean space, the Newton copolygon associated with such a power series comes with a stratification into $k$-skeletons with $0 \leq k \leq 2$: the $0$-skeleton is made of the vertices of the copolygon, the $1$-skeleton consists in all edges linking two vertices of the copolygon, while the $2$-skeleton is the union of the $2$-dimensional facets of the copolygon. We will see below (see Proposition \ref{propexcepvalues}) that the $0$-skeleton can be characterised in terms of exceptional values (as defined in Definition \ref{defexcepvalues} below).
\end{rem}

\begin{defn} \label{defexcepvalues}
We keep the same notation as in Definition \ref{DefVf}. A pair $(\xi_{1}, \xi_{2})$ of real numbers is called an {\it exceptional value for $V_{f}$} if there exists three elements $(i_{1}, j_{1})$, $(i_{2}, j_{2})$ and $(i_{3}, j_{3})$ in $\AAA$ such that 
\[\displaystyle \forall \ k \in \{1,2,3\}, \ V_{f}(\xi_{1}, \xi_{2}) = i_{k}\xi_{1} + j_{k}\xi_{2} + v(c_{i_{k}j_{k}}) \ . \]
In other words, exceptional values are pre-images of elements in the image of $V_{f}$  that are reached from (at least) three different elements of $\AAA$.
\end{defn}

\begin{prop} \label{propexcepvalues}
Let $f \in \Z_{p}[[x_{1}, x_{2}]]$ be a $p$-adic power series in $2$ variables. Let $\NNN_{f}^{c}$ be its Newton copolygon and $V_{f} : \R^{2} \to \R$ be the associated copolygon function (as in Definition \ref{DefVf}).
\begin{enumerate}
\item A pair $(\xi_{1}, \xi_{2})$ of real numbers is an exceptional value of $V_{f}$ if, and only if, $(\xi_{1}, \xi_{2}, V_{f}(\xi_{1}, \xi_{2}))$ is in the $0$-skeleton of $\NNN_{f}^{c}$.
\item For any element $(\alpha_{1}, \alpha_{2}) \in \overline{\mathfrak{m}}^{2} $, we have $v(f(\alpha_{1}, \alpha_{2})) \geq V_{f}(v(\alpha_{1}), v(\alpha_{2}))$.
\item For any element $(\alpha_{1}, \alpha_{2}) \in \overline{\mathfrak{m}}^{2}$, the following are equivalent : 
\begin{itemize}
\item there exists $u \in \overline{\Z}_{p}^{\times}$ such that $(\alpha_{1}, u\alpha_{2})$ is a zero of $f$;
\item $(v(\alpha_{1}), v(\alpha_{2}))$ is an exceptional value for $V_{f}$.
\end{itemize}
\end{enumerate}
\end{prop}
\begin{proof} \noindent 
\begin{itemize}
\item By definition, the exceptional values of $V_{f}$ are exactly the two first coordinates of points where (at least) three half-planes of the convex body that defines $\NNN_{f}^{c}$ meet, i.e. they are the two first coordinates of a vertex of $\NNN_{f}^{c}$. The third coordinate of such a vertex is (by definition of $\NNN_{f}^{c}$) the image by $V_{f}$ of the two first coordinates, hence the first statement is proven.
\item Let $(\alpha_{1}, \alpha_{2}) \in \overline{\mathfrak{m}}^{2}$: then $f(\alpha_{1}, \alpha_{2})$ is a convergent series. Moreover, if we write as above $f(x,y) = \displaystyle \sum_{i,j \geq 0} c_{ij}x^{i}y^{j}$ with $c_{ij} \in \Z_{p}$ and let $\AAA$ denote the set of all pairs $(i,j)$ of indices such that $c_{ij} \not= 0$, then we have (since $v$ is a valuation) : 
\begin{equation} \label{eq1a}
\displaystyle 
v(f(\alpha_{1}, \alpha_{2})) = v\left(\sum_{i,j \geq 0}  c_{ij} \alpha_{1}^{i}\alpha_{2}^{j} \right) \geq \min \{v(c_{ij}) + i v(\alpha_{1}) + j v(\alpha_{2}), \  (i,j) \in \AAA\}\ ,
\end{equation}
the right hand side being by definition nothing but $V_{f}(v(\alpha_{1}), v(\alpha_{2}))$. Hence, we have the claimed inequality.
\item Assume that $(\alpha_{1}, \alpha_{2}) \in \overline{\mathfrak{m}}^{2}$ is a zero of $f$. Then the inequality in statement (2) above must be strict, which implies by contraposition that $(v(\alpha_{1}), v(\alpha_{2}))$ must be an exceptional value for $V_{f}$. Conversely, assume that $(\alpha_{1}, \alpha_{2}) \in \overline{\mathfrak{m}}^{2}$ is such that $(v(\alpha_{1}), v(\alpha_{2}))$ is an exceptional value for $V_{f}$ and set $g(x) := f(\alpha_{1}, x)$. Then $g$ is a $p$-adic power series in one variable $x$ for which $v(\alpha_{2})$ is exceptional in the sense of \cite[Definition 12]{JL1}. According to \cite[Proposition A.2.5.]{JL1}, this implies that there exists a unit $u \in \overline{\Z}_{p}^{\times}$ such that $g(u \alpha_{2}) = 0$, i.e. such that $f(\alpha_{1}, u\alpha_{2}) = 0$, and this ends the proof of statement (3), as well as the proof of the proposition. 
\end{itemize}
\end{proof}
\begin{rem}
The proof above shows that the statements (2) and (3) of Proposition \ref{propexcepvalues} actually hold for any $(\alpha_{1}, \alpha_{2}) \in \overline{\Z}_{p}$ such that the series $f(\alpha_{1}, \alpha_{2})$ converges.
\end{rem}

\subsubsection{Two explicit examples}
For a better understanding of how things work, we offer two concrete examples of Newton copolygons associated with explicit $p$-adic power series in $2$ variables. The second one connects the content of this subsection with the study of the $p^{\infty}$-torsion points of $2$-dimensional Lubin-Tate formal groups mentioned in the previous subsection.
	\begin{exmp} \label{ex1} Assume that $p = 2$ and consider
	\[ \displaystyle f(x_{1},x_{2}) : =2x_{1}x_{2}+x_{1}^{4}+x_{2}^{5} \in \mathbb{Z}_{2}[[x_{1},x_{2}]] \ . \]
The set $\AAA$ consists here in $3$ elements, namely the pairs $(1,1)$, $(4,0)$ and $(0,5)$, with $v(c_{11}) = 1$ and $v(c_{40}) = 0 = v(c_{05})$. Consequently, the convex body whose boundary defines the Newton copolygon $\NNN_{f}^{c}$ is given by the following set of three inequations : 
\[\displaystyle \eta \leq \xi_{1} + \xi_{2} + 1, \ \eta \leq 4 \xi_{1} \ \text{ and } \eta \leq 5\xi_{2} \ .\]
The 3 half-spaces defined by these inequations, as well as their intersections, can be seen on Figure \ref{fig:figA} below, while Figure \ref{fig:figB} shows their projection onto the plane  defined by $\eta = 0$. 
	
\begin{figure}[h]
	\centering
	\begin{minipage}[b]{0.54\linewidth} \centering
		\tikzset
		{%
			plane 1/.style={thick,blue,fill=cyan!20,fill opacity=0.9},            % \eta=\xi_1+\xi_2+1
			plane 2/.style={thick,green!40!black,fill=lime!20,fill opacity=0.9}, % \eta=4\xi_1
			plane 3/.style={thick,brown,fill=teal!20,fill opacity=0.9},         % \eta=5\xi_2
			inter/.style  ={thick,red},                                           % intersection lines
		}
		\begin{tikzpicture}
		[%
		x={(-0.4cm,-0.2cm)},y={(0.8cm,-0.3cm)},z={(0cm,0.4cm)},%
		line cap=round,line join=round%
		]
		% z=x+y+1
		\coordinate (X1) at (0,0,1);
		\coordinate (X2) at (4,0,5);
		\coordinate (X3) at (4,4,9);
		\coordinate (X4) at (0,4,5);
		% z=4x
		\coordinate (O)  at (0,0,0);
		\coordinate (2Y) at (4,0,16);
		\coordinate (3Y) at (4,4,16);
		\coordinate (4Y) at (0,4,0);
		% z=5y
		\coordinate (2Z) at (4,0,0);
		\coordinate (3Z) at (4,4,20);
		\coordinate (4Z) at (0,4,20);
		% intersection points
		\coordinate (a1)  at (4,3.2,16);
		\coordinate (b1)  at (4,1.25,6.25);
		\coordinate (c1)  at (5/11,4/11,20/11);
		\coordinate (d1)  at (5/3,4,20/3);
		\coordinate (e1)  at (0,0.25,1.25);
		\coordinate (f1)  at (1/3,0,4/3);
		% projection points
		\coordinate (A)  at (4,4,0);
		\coordinate (B)  at (5/11,4/11,0);
		\coordinate (C)  at (4,1.25,0);
		\coordinate (D)  at (5/3,4,0);
		% axes and dashed lines
		\draw[dashed] (4,0,0) -- (2Y);
		\draw[dashed] (0,4,0) -- (4Z);
		\draw[thick,-latex] (O) -- (5,0,0)  node [left]  {$\xi_1$};
		\draw[thick,-latex] (O) -- (0,5,0)  node [right] {$\xi_2$};
		\draw[thick,-latex] (O) -- (0,0,15) node [above] {$\eta$};
		% planes and intersection lines
		\draw[plane 1] (X1) -- (f1)  -- (c1)  -- (e1)  -- cycle;
		\draw[plane 3] (O)  -- (a1)  -- (3Z) -- (4Z) -- cycle;
		\draw[plane 1] (e1)  -- (X4) -- (d1)  -- (c1)  -- cycle;
		\draw[plane 2] (O)  -- (2Y) -- (3Y) -- (4Y) -- cycle;
		\draw[plane 1] (f1)  -- (X2) -- (b1)  -- (c1)  -- cycle;
		\draw[plane 3] (O)  -- (2Z) -- (3Z) -- (a1)  -- cycle;
		\draw[inter]   (O)  -- (a1);
		\draw[plane 1] (b1)  -- (c1)  -- (d1)  -- (X3) -- cycle;
		\draw[inter]   (b1)  -- (c1)  -- (d1);
		% projection
		\draw[fill=gray!20] (B) -- (D)  -- (A) -- (C)  -- cycle;
		\draw[fill=gray!30] (O) -- (4Y) -- (D) -- (B)  -- cycle;
		\draw[fill=gray!40] (O) -- (B)  -- (C) -- (2Z) -- cycle;
		% more dashed lines
		\draw[dashed] (A) -- (3Z);
		\draw[dashed] (B) -- (c1);
		\draw[dashed] (C) -- (b1);
		\draw[dashed] (D) -- (d1);
		% labels
		\node[blue]           at (X4) [right] {$\eta=\xi_1+\xi_2+1$};
		\node[lime!40!black] at (2Y) [above] {$\eta=4\xi_1$};
		\node[teal]          at (4Z) [above] {$\eta=5\xi_2$};
		\fill[red] (c1) circle (1pt) node [below right] {$\left(\frac{5}{11},\frac{4}{11},\frac{20}{11}\right)$};
		\end{tikzpicture}
		\caption{The Newton copolygon from Example \ref{ex1}}\label{fig:figA}
	\end{minipage}
	\begin{minipage}[b]{0.44\linewidth} \centering
		\tikzset
		{%
			plane 1/.style={fill=cyan!20  ,fill opacity=0.9}, % \eta=\xi_1+\xi_2+1
			plane 2/.style={fill=teal!20 ,fill opacity=0.9}, % \eta=4\xi_1
			plane 3/.style={fill=lime!20,fill opacity=0.9}, % \eta=5\xi_2
		}
		\begin{tikzpicture}[thick,line cap=round,line join=round]
		% z=4x
		\coordinate (O)  at (0,0);
		\coordinate (4Y) at (0,4);
		% z=5y
		\coordinate (2Z) at (4,0);
		% projection points
		\coordinate (A)  at (4,4);
		\coordinate (B)  at (5/11,4/11);
		\coordinate (C)  at (4,1.25);
		\coordinate (D)  at (5/3,4);
		% axes
		\draw[thick,-latex] (O) -- (5,0) node [right] {$\xi_1$};
		\draw[thick,-latex] (O) -- (0,5) node [above] {$\xi_2$};
		% projection
		\draw[plane 1] (B) -- (D)  -- (A) -- (C)  -- cycle;
		\draw[plane 2] (O) -- (4Y) -- (D) -- (B)  -- cycle;
		\draw[plane 3] (O) -- (B)  -- (C) -- (2Z) -- cycle;
		% labels
		\node at (2.5,2.5) {$\eta=\xi_1+\xi_2+1$};
		\node at (0.7,3.5) {$\eta=4\xi_1$};
		\node at (3  ,0.5) {$\eta=5\xi_2$};
		\fill[red] (B) circle (1pt) node [above right] {$\left(\frac{5}{11},\frac{4}{11}\right)$};
		\end{tikzpicture}
		\caption{Its projection on the $\eta = 0$ plane}\label{fig:figB}
	\end{minipage}
	%\caption{Newton Copolygon $\mathcal{N}_{\tilde{f}}^c$} \label{fig:figC}
\end{figure}
Note that, similarly to what happens in the $1$-variable case (see for instance \cite[Fig. 1]{JL1}), the copolygon is concave down, when viewed from above. Moreover, the $0$-skeleton of $\NNN_{f}^{c}$ is made of the unique point where the three half-spaces intersect, namely $(\frac{5}{11},\frac{4}{11}, \frac{20}{11})$.
	\end{exmp}
		
	%The teal-coloured region is where $\eta=4 \xi_1$; for instance, $(\frac{1}{2},1,\frac{5}{2})$ is in the teal-coloured region, and if $v(\alpha_1)=\frac{1}{2}$ and $v(\alpha_2)=1$, then we get the conclusion that $v(f(\alpha_1,\alpha_2))=v(2\alpha_1\alpha_2+\alpha_1^4+\alpha_2^5)=4v(\alpha_1)=2$.  Similarly, the lime-coloured region is where $\eta=5 \xi_2$, and the cyan-coloured region is where $\eta=\xi_1+\xi_2+1$. In three-dimensional picture, the copolygon is concave down, when viewed from above, just as in the situation for the copolygon in a single variable.  
		%The following example demonstrates existence of a $2$-dimensional Lubin-Tate formal group $F$ associated to $[p]_F(X)=(px_1+x_2^{p^2},px_2+x_1^{p^3}),~ X=(x_1,x_2)$.
	\begin{exmp} \label{e3.2} In this second example, we keep $p = 2$ and consider the pair of $2$-adic power series in $2$ variables associated with the dynamical system $\DDD_{(2,3)}$ defined by \eqref{defDDD}, namely
	\[\displaystyle f(x_{1}, x_{2}) = (2x_{1} + x_{2}^{4}, 2x_{2} + x_{1}^{8})  = (f_{1}(x_{1}, x_{2}), f_{2}(x_{1}, x_{2}))\ . \]
Recall that this dynamical system is of interest as the simplest approximation of the multiplication-by-$2$ endomorphism of the $2$-dimensional Lubin-Tate formal group $F$ over $\Z_{2}$ associated with $(h_{1}, h_{2}) = (2,3)$. Hence, we are interested in the common zeros of these $2$-adic power series, as candidates for the $2$-torsion points of $F$.

To do this, we consider the Newton copolygons of $f_{1}$ and $f_{2}$. Starting with $f_{1}$, we see that the only two non-zero coefficients correspond to the pair of indices $(1,0)$ and $(0,4)$, with respective valuations $v(c^{(1)}_{10}) = 1$ and $v(c^{(1)}_{04}) = 0$. This means that the two half-spaces involved in the definition of the Newton copolygon $\NNN_{f_{1}}^{c}$ are given by the following inequalities:
\[\displaystyle \eta \leq \xi_{1} + 1 \ \text{ and } \eta \leq 4 \xi_{2} \ . \]
The intersection of these two half-spaces with $\NNN_{f_{1}}^{c}$, which must contain all possible zeroes of $f_{1}$ according to Proposition \ref{propexcepvalues}, is given by the line whose equation is $\xi_{1} + 1 = 4 \xi_{2}$: it is the gentle-sloped line drawn with dashes on Figure \ref{f3.4}.

Similarly, the half-spaces involved in the definition of the Newton copolygon $\NNN_{f_{2}}^{c}$ associated with $f_{2}$ are given by the following inequalities:
\[ \displaystyle \eta \leq \xi_{2} + 1 \ \text{ and } \eta \leq 8\xi_{1} \ . \]
The corresponding locus for possible zeroes for $f_{2}$ is hence the line whose equation is $ \xi_{2} = 8 \xi_{1} - 1$: it corresponds to the dotted line on Figure \ref{f3.4}.
		\begin{figure}[h]
			\centering
			\begin{tikzpicture}[>=latex]
			\begin{axis}[mystyle/.style={thick},
			axis x line=center,
			axis y line=center,
			xtick={1},
			ytick={1},
			xlabel style={right},
			ylabel style={above},
			x tick label style={anchor=south,above},
			yticklabel style={black},
			xlabel={$\xi_1$},
			ylabel={$\xi_2$},
			xmin=-1,
			xmax=2,
			ymin=-.5,
			ymax=2]
			\addplot[mystyle,dashed,blue]coordinates{(0,1/4)(10,11/4)};
			\addplot[mystyle,dotted,blue]coordinates{(0,-1)(1,8)};
			\addplot coordinates {(5/31,9/31)};
			\node[above, anchor=west] at (axis cs:0.1,0.6) {(5/31,9/31)};
			\end{axis}
			\end{tikzpicture}
             \caption{Intersecting the zero locus of $f_{1}$ and $f_{2}$ from Example \ref{e3.2}}\label{f3.4}
		\end{figure}
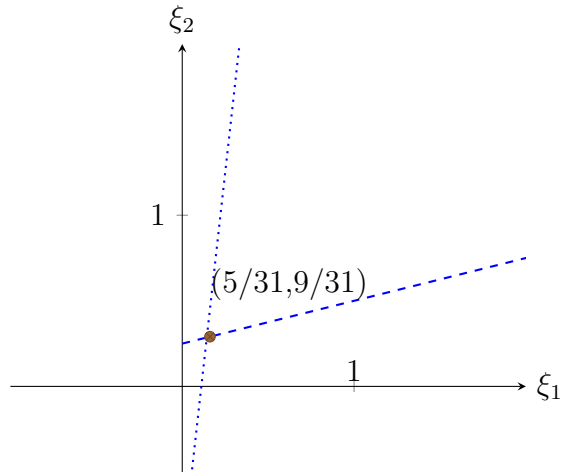
		
The intersection point of the two lines on Figure \ref{f3.4} has for coordinates $(\xi_{1}, \xi_{2}) = \left(\frac{5}{31}, \frac{9}{31} \right)$. According to statement (3) of Proposition \ref{propexcepvalues}, this means that any non-trivial zero $(\alpha_{1}, \alpha_{2})$ of $f$ must satisfy $v(\alpha_{1}) = \frac{5}{31}$ and $v(\alpha_{2}) = \frac{9}{31}$, where $v$ denotes here the $2$-adic valuation, and that any element $(a_{1}, a_{2}) \in \Z_{2}^{2}$ satisfying $v(a_{1}) = \frac{5}{31}$ and $v(a_{2}) = \frac{9}{31}$ is of the form $(\alpha_{1}, u \alpha_{2})$ for some unit $u \in \Z_{2}^{\times}$ and some root $(\alpha_{1}, \alpha_{2})$ of $f$ in $\Z_{2}^{2}$. Note that the shape of $v(a_{1})$ and $v(a_{2})$ moreover ensures that $f$ has $31$ non-trivial zeroes in $\Z_{2}$, hence $32 = 2^{5}$ zeroes in total as $(0,0)$ is also an option. Furthermore, if we fix any $31$-st root of unity $\zeta$, then $(\zeta\alpha_{1}, \zeta^{4}\alpha_{2})$ provides a zero of $f$ whenever $(\alpha_{1}, \alpha_{2})$ is a zero of $f$. All these zeros are actually $2$-torsion points of $F$; hence, we obtain that the number of $2$-torsion points of $F$ in $\Z_{2}$ is $2^{5} = 2^{h_{1} + h_{2}}$.
	\end{exmp}
	
\section{Galois extensions attached to a 2-dimensional Lubin-Tate formal group} \label{s4}
In this last section, we investigate some properties of the field extension $\Q_{p}(F[p^{n}])/\Q_{p}$ introduced in Section \ref{s31}. More precisely, we are interested in understanding the ramification properties attached to such an extension. Recall that $F$ is the $2$-dimensional Lubin-Tate formal group over $\Z_{p}$ associated with $\vec{h} = (h_{1}, h_{2})$ and that $n \geq 1$ is any positive integer.

The first thing to check is that the extension $\Q_{p}(F[p^{n}])/\Q_{p}$ built by adjunction of the coordinates of the $p^{n}$-torsion points of $F$ is Galois. This is the point of the next statement. 
	\begin{thm} \label{p4.3}
	For any $n \geq 1$, the field $\Q_{p}(F[p^{n}])$ generated over $\Q_{p}$ by the $p^{n}$-torsion points of $F$ defines a Galois extension of $\Q_{p}$.
	\end{thm}		
	\begin{proof}
	To ease notation, we set $K_{n} := \Q_{p}(F[p^{n}])$ and let $K_{n}^{c}$ denote the Galois closure of $K_{n}$ over $\Q_{p}$, i.e. the smallest Galois extension of $\Q_{p}$ that contains $K_{n}$. To prove that $K_{n} = K_{n}^{c}$, it is enough to prove that any element of $F[p^{n}]$ has all its Galois conjugates in $K_{n}$, which will be the case if we prove that: 
	\[\displaystyle \forall \ \alpha \in F[p^{n}], \ \forall \ \tau \in \mathrm{Gal}(K_{n}^{c}/\Q_{p}), \ \tau(\alpha) \in F[p^{n}] \ . \] 
	Let $\alpha \in F[p^{n}]$ and $\tau \in \mathrm{Gal}(K_{n}^{c}/\Q_{p})$.
	Recall that the natural structure of $\Z_{p}$-module carried by $F[p^{n}]$ (see Lemma \ref{l4.1}) is such that the scalar multiplication is given by:
	\[\displaystyle \forall \ u \in \Z_{p}, \ u \cdot \alpha := [u]_{F}(\alpha) \ .\]
	This implies in particular that $p^{n} \in \Z \subset \Z_{p}$ acts trivially on $\alpha$. Since $p^{n} = \tau(p^{n})$, we directly have that 
	\[\displaystyle p^{n} \cdot \tau(\alpha) = \tau(p^{n}) \cdot \tau(\alpha) = \tau(p^{n} \cdot \alpha) = \tau (\vec{0}) = \vec{0} \ .\]
	As $\tau$ induces a permutation of the elements of $F[p^{n}]$, we obtain that $\tau(K_{n}) = K_{n}$ - by double inclusion. Hence, $\tau$ induces a $\Q_{p}$-automorphism of $K_{n}$, and the proof is complete.
	\end{proof}
	
\subsection{Building (non) abelian extensions from the $p^{\infty}$-torsion points of $F$}	\noindent 
Before going further, let us explain, by analogy with the $1$-dimensional case, why these Galois extensions have a priori no reason to be abelian extensions.\\
Let $h \geq 2$ be any fixed integer. Let $G$ be a classical ($1$-dimensional) Lubin-Tate formal group over $\Z_{p}$ such that the $p$-adic power series $f(x) = px + x^{p^{h}}$ defines the multiplication-by-$p$ endomorphism $[p]_{G}$ of $G$ over $\Z_{p}$. (All this makes sense in the classical case by \cite{JL2}.) Although it is well-defined over $\Z_{p}$, it does not provide a Lubin-Tate formal group over $\Q_{p}$, but over the unramified extension $U_{h}$ of degree $h$ over $\Q_{p}$. Note that $p$ remains a prime element in $U_{h}$, and the corresponding residue field $\UUU_{h}/p\UUU_{h}$ is isomorphic to the finite field $\FF_{p^{h}}$. By Hensel's lemma, this ensures (as $\FF_{p^{h}}^{*}$ is cyclic of order $p^{h} -1$) that $U_{h}$ contains all $p^{h} -1$ roots of unity in $\QQ_{p}$.\\
Now remark that $U_{h}/\Q_{p}$ is an abelian extension (since finite and unramified)	, and $\mathrm{Gal}(U_{h}/\Q_{p})$ is the cyclic (abelian) group of order $h$. The problem comes from the next stage: indeed, if we set $U_{\infty, G} := \displaystyle \bigcup_{n \geq 1} \Q_{p}[G[p^{n}]]$, then it defines a huge totally ramified abelian extension of $U_{h}$, but it is definitely {\bf not} an abelian extension of $\Q_{p}$, since the maximal abelian extension of $\Q_{p}$ is much smaller (as it is generated by the cyclotomic extensions of $\Q_{p}$). More precisely, note that the Galois group of $U_{\infty, G}/U_{h}$ is actually isomorphic to $\UUU_{h}^{\times}$, hence the Galois group of $U_{\infty, G}/\Q_{p}$ is isomorphic to a semi-direct product $\UUU_{h}^{\times} \rtimes \Z/p^{h}\Z$, which has no reason to be a direct product in general. This suggests that in higher dimension, there is little chance that our Galois extensions $K_{n}/\Q_{p}$ are abelian in general.

Back to our $2$-dimensional Lubin-Tate formal group $F$ of height $p^{h_{1} + h_{2}}$, we denote by
\[ \mathrm{Tors}(F) := \displaystyle \bigcup_{n \geq 1} F[p^{n}] \]
the $p^{\infty}$-torsion part of $F$. Note that $\ker([p]_{F})$ provides $p^{h_{1} + h_{2}}$ points of an affine $2$-space over $\Q_{p}$, and that adjoining them to $\Q_{p}$ automatically includes $U_{h_{1}+h_{2}}$ in the field built this way. This motivates the study of the extension of $U_{h_{1} + h_{2}}$ generated by $\mathrm{Tors}(F)$, as a candidate for a nice abelian extension. The next statement, which is one of the main results of this paper, shows that this hope is reasonnable.
	\begin{thm} \label{t4.2}
	The extension $U_{h_{1} + h_{2}}(\mathrm{Tors}(F))/U_{h_{1} + h_{2}}$ is abelian.
	\end{thm}
\begin{proof}
To ease notation, we set $h := h_{1} + h_{2}$ and $U := U_{h}(\mathrm{Tors}(F))$. Our goal is then to prove that $\mathrm{Gal}(U/U_{h})$ is an abelian group. First, let us explain how elements of $\UUU_{h}$ define endomorphisms of our Lubin-Tate group $F$, seen as a Lubin-Tate group over $\UUU_{h}$. Note that it is enough to understand what happens for $p^{h}-1$ roots of unity, since the latter generate the $\Z_{p}$-module $\UUU_{h}$.

\noindent Let $\gamma$ be a $(p^{h}-1)$-th root of $1$ in $U_{h}$. We define $\{\gamma\}$ by the following formula:
\begin{equation} \label{defgammaendo}
\{\gamma\}(x_{1}, x_{2}) := (\gamma x_{1}, \gamma^{p^{h_{2}}}x_{2}) \ .
\end{equation}
Let us check that $\{\gamma\}$ defines an endomorphism of $F$. First note that, by construction, we have $\{\gamma\}(X) \equiv 0 \textrm{ mod degree 1 terms}$; hence, we are left to check that $\{\gamma\}$ commutes with $F$. To do this, let us substitute $\{\gamma\}(x_{1}, x_{2})$ to $(x_{1}, x_{2})$ in \eqref{e4} and \eqref{e5}. As $\gamma^{p^{h}} = \gamma$, we obtain that
\[ \displaystyle 
\begin{array}{rcl}
L_{1}(\{\gamma\}(X)) & = & L_{1}(\gamma x_{1}, \gamma^{p^{h_{2}}}x_{2}) \\
& = & \displaystyle \gamma x_{1} + \sum_{k \geq 1} p^{-2k} \gamma^{p^{kh}}x_{1}^{p^{kh}} + \sum_{k \geq 0} p^{-(2k+1)}\left(\gamma^{p^{h_{2}}}x_{2} \right)^{p^{h_{1} + kh}}\\
& = & \displaystyle \gamma x_{1} + \sum_{k \geq 1 p^{-2k}} \gamma x_{1}^{p^{kh}} + \sum_{k \geq 0} p^{-(2k+1)} \gamma^{(k+1)p^{h}}x_{2}^{p^{h_{1} + kh}} \\
& = & \displaystyle \gamma \left(x_{1} + \sum_{k \geq 1} p^{-2k} x_{1}^{p^{kh}} + \sum_{k \geq 0} p^{-(2k+1)} x_{2}^{p^{h_{1} + kh}} \right) \\
& = & \gamma L_{1}(x_{1}, x_{2}) \ , 
\end{array}\]
and 
\[\displaystyle 
\begin{array}{rcl}
L_{2}(\{\gamma\}(X)) & = & L_{2}(\gamma x_{1}, \gamma^{p^{h_{2}}} x_{2}) \\
& = & \displaystyle \gamma^{p^{h_{2}}} x_{2} + \sum_{k \geq 1} p^{-2k} \left(\gamma^{p^{h_{2}}}x_{2}\right)^{p^{kh}} + \sum_{k \geq 0} p^{-(2k+1)} (\gamma x_{1})^{p^{h_{2} + kh}} \\
& = &  \displaystyle \gamma^{p^{h_{2}}} x_{2} + \sum_{k \geq 1} p^{-2k} \left(\gamma^{p^{h_{2}}}x_{2}\right)^{p^{kh}} + \sum_{k \geq 0} p^{-(2k+1)} (\gamma^{p^{kh}})^{p^{h_{2}}} x_{1}^{p^{h_{2} + kh}} \\
& = & \displaystyle \gamma^{p^{h_{2}}} \left( x_{2} + \sum_{k \geq 1} p^{-2k} x_{2}^{p^{h_{2} + kh}} + \sum_{k \geq 0} p^{-(2k+1)} x_{1}^{p^{h_{2} + kh}}\right)\\
& = & \gamma^{p^{h_{2}}} L_{2}(x_{1}, x_{2}) \ .
\end{array}\]
We have hence proven that 
\[\displaystyle L(\{\gamma\}(X)) = M_{\gamma} L(X) \ \text{ with } M_{\gamma} = \left(\begin{array}{cc} \gamma & 0 \\ 0 & \gamma^{p^{h_{2}}}\end{array}\right) \in \mathrm{M}_{2}(U_{h}^{\times}) \ .\]
This is enough to ensure that $\{\gamma\}$ is an endomorphism of $F$. Indeed, since we have 
\[\displaystyle L(\{\gamma\}(X)) + L(\{\gamma\}(Y)) = M_{\gamma}(L(X) + L(Y)) \ , \]
then we have 
\begin{equation} \label{gamma1}
\displaystyle L((F \circ \{\gamma\})(X,Y)) = L(\{\gamma\}(X)) + L(\{\gamma\}(Y))) = M_{\gamma} (L(X) + L(Y))
\end{equation}
and 
\begin{equation} \label{gamma2}
\displaystyle  L((\{\gamma\} \circ F)(X,Y)) = M_{\gamma} L(F(X,Y)) = M_{\gamma} (L(X) + L(Y)) \ .
\end{equation}
Comparing \eqref{gamma1} and \eqref{gamma2} shows that $L((F \circ \{\gamma\})(X,Y)) = L((\{\gamma\} \circ F)(X,Y))$, hence that 
\[F \circ \{\gamma\} = \{\gamma\} \circ F \ , \]
as announced.

In particular, $\{\gamma\}$ commutes with multiplication by $p$-adic integers and, as it satisfies 
\[ \{\gamma\}^{\circ (p^{h} - 1)} = \mathrm{Id}_{F} \ , \]
we can see $\Z_{p}[\gamma] \cong \Z_{p}[\{\gamma\}]$ as a subring of the $\UUU_{h}$-endomorphism ring $\mathrm{End}_{\UUU_{h}}(F)$ of $F$. (Recall that $\Z_{p}[\{\gamma\}]$ denotes the smallest subring of $\mathrm{End}_{\UUU_{h}}(F)$ that contains $\Z_{p}$ and $\{\gamma\}$.) Hence, we can embed the group of ($p^{h}-1$)-th roots of $1$ in $U_{h}$ as a subgroup of $\mathrm{Aut}_{\UUU_{h}}(F)$, the latter being the group of $\UUU_{h}$-automorphisms of $F$. This proves in particular that $\mathrm{Aut}_{\UUU_{h}}(F) = \mathrm{End}_{\UUU_{h}}(F)^{\times}$ has a cyclic subgroup of order $p^{h}-1$.

Assume now that $\gamma$ is a given primitive $(p^{h}-1)$-th root of $1$ in $\UUU_{h}$. Considering $\UUU_{h}$ as a free $\Z_{p}$-module with basis $\{\gamma^{i}, 0 \leq i \leq p^{h} -2\}$, we can define a $\Z_{p}$-linear map $\Gamma : \UUU_{h} \to \mathrm{End}_{\UUU_{h}}(F)$ by setting $\Gamma(\gamma^{i}) := \{ \gamma \}^{\circ i}$ for any $i \in \{0, \ldots , p^{h}-2\}$. It is straightforward to check that $\Gamma$ is actually an injective ring homomorphism whose image is $\Z_{p}[\{\gamma\}]$, so we have identified elements of $\UUU_{h}$ with explicit endomorphisms of the Lubin-Tate group $F$.

This shows in particular that $\mathrm{End}_{\UUU_{h}}(F)$ contains a subring that is also a free $\Z_{p}$-subalgebra of rank $h$. Let us introduce here the Tate module $\mathrm{T}_{p}(F) =\displaystyle \lim_{\leftarrow} F[p^{n}]$ of the formal group $F$. Since $F$ is of height $h$, $\mathrm{T}_{p}(F)$ is a free $\Z_{p}$-module of rank $h$; hence, $\Gamma$ turns it into a free $\UUU_{h}$-module of rank $1$. 

Now, let $\sigma$ be an element of $\mathrm{Gal}(U/U_{h})$, as well as the endomorphism of $\mathrm{T}_{p}(F)$ it defines. Since $\{\gamma\}(X) = M_{\gamma} X$ with $M_{\gamma} \in \mathrm{M}_{2}(\UUU_{h}^{\times})$, we directly obtain that $\sigma$ and $\{\gamma\}$ commute. The same holds for any $(p^{h}-1)$-th root of $1$, since the associated endomorphism of $F$ will just be the corresponding power of $\{\gamma\}$. The identification made above regarding the structure of $\UUU_{h}$-module of $\mathrm{T}_{p}(F)$ hence ensures that $\sigma$ is actually a $\UUU_{h}$-endomorphism of $\mathrm{T}_{p}(F)$. But $\mathrm{T}_{p}(F)$ has rank $1$ as $\UUU_{h}$-module, hence all its endomorphisms are exactly multiplication by elements of $\UUU_{h}$: hence, we have proven that $\sigma$ commute with all $\UUU_{h}$-endomorphisms of $\mathrm{T}_{p}(F)$.

In other words, this implies that the action of $\mathrm{Gal}(U/U_{h})$ on $T_{p}(F)$ defines a rank $1$ representation of $\mathrm{Gal}(U/U_{h})$ over $\UUU_{h}$. It is straightforward to check that this representation is faithful since $\mathrm{Gal}(U/U_{h})$ maps onto $\mathrm{Aut}_{\UUU_{h}}(F)$ in a $1-1$ way. Consequently, any element of $\mathrm{Gal}(U/U_{h})$ acts on $T_{p}(F)$ as an element of $\UUU_{h}$; hence, these actions commute to each other (since $\{\gamma\} \circ \{\mu\} = \{\mu\} \circ \{\gamma\}$ for any elements $\gamma, \mu$ in $\UUU_{h}$). But $U = U_{h}(\mathrm{Tors}(F))$, so we have proven that any two elements of $\mathrm{Gal}(U/U_{h})$ commute to each other, which ends the demonstration.
\end{proof} 
 
    \begin{rem} Considering what happens for $1$-dimensional Lubin-Tate formal groups, we conjecture that $U_{h_{1} + h_{2}}(\mathrm{Tors}(F))$ should lie inside the maximal abelian extension of $U_{h_{1} + h_{2}}$. Indeed, given any unit $u \in \Z_{p}^{\times}$ and any integer $r \geq 1$, let $G_{u,r}$ denote the $1$-dimensional Lubin-Tate formal group whose multiplication-by-$p$ endomorphism is given by the $p$-adic power series $upx + x^{p^{r}}$. Let $U_{r}^{LT}$ denote the extension of $U_{r}$ generated by the $p^{\infty}$-torsion part of $G_{u,r}$ and let $U_{r}^{nr}$ denote the maximal unramified extension of $U_{r}$ inside $\QQ_{p}$. Then we know from \cite{JL2} that the compositum $U^{LT}_{r}U^{nr}_{r}$ is the maximal abelian extension of $U_{r}$. Our conjecture is hence a consequence of what would happens if an analogue phenomenon for higher-dimensional Lubin-Tate groups holds.
    \end{rem}

 \subsection{Ramification properties of $p^{\infty}$-torsion generated extensions}   
 
 We devote the end of this paper to an investigation of the ramification properties of the extensions $U_{h_{1} + h_{2}}(F[p^{n}])/U_{h_{1} + h_{2}}$ and $U_{h_{1} + h_{2}}(\mathrm{Tors}(F))/U_{h_{1} + h_{2}}$ introduced in the previous subsection. Let us already mention that we do not pretend to provide exhaustive results in this paper, but only initiate this study, which will be continued in a forthcoming work of the authors \cite{AS2}.
 
 Let us start by proving a useful arithmetical lemma.
     \begin{lem}\label{l4.3}
     Let $s \geq 2$ and $t \geq 2$ be coprime integers, with $s$ odd. Then 
     \[\displaystyle \gcd\left(\frac{p^{s} -1}{2}, \frac{p^{t} +1}{2} \right) = 1 \ .\]
     \end{lem}
     \begin{proof}
     Let $d$ be the greatest common divisor of $\displaystyle \frac{p^{s} - 1}{2}$ and $\displaystyle \frac{p^{t}+1}{2}$, and assume by contradiction that $d > 1$. Then there exists a prime integer $\ell$ that divides $d$. Such a prime $\ell$ then satisfies 
     \[\displaystyle p^{s} \equiv 1 \ \mathrm{mod} \ \ell \ \text{ and } p^{t} \equiv - 1 \ \mathrm{mod} \ \ell \ . \]
  If $x$ denotes the order of $p$ in the (cyclic) multiplicative group $\FF_{\ell}^{*}$, then Lagrange's theorem implies that $x$ divides simultanously $s$ and $2t$. As $s$ is odd, $x$ must be odd too, hence coprime to $2$, thus it must divide $t$. In other words, $x$ is a common divisor of $s$ and $t$, which are assumed to be coprime: this means that $x = 1$, i.e. that $p \equiv 1 \ \mathrm{mod} \ \ell$. This implies that we must have $p^{t} \equiv 1 \ \mathrm{mod} \ \ell$, but we already have $p^{t} \equiv -1 \ \mathrm{mod} \ \ell$, so we obtain that $1 \equiv -1 \ \mathrm{mod} \ \ell$, which is possible if, and only if, $\ell = 2$. In other words, given the definition of $\ell$, we must have 
  \[\displaystyle p^{s} \equiv 1 \ \mathrm{mod} \ 4 \ \text{ and } \ p^{t} \equiv -1 \ \mathrm{mod} \ 4 \ . \]
  Since $s$ is odd, the first congruence is only possible if $p$ is congruent to $1$ modulo $4$. But this condition is not compatible with the second congruence, which would imply that $1$ and $-1$ are congruent modulo $4$. This is a contradiction: hence, $d = 1$ and the proof is complete.
     \end{proof}
     
     The next statement is at the core of this subsection, since it provides explicit general formulae for the $p$-adic valuation of the $p^{n}$-torsion points of the $2$-dimensional Lubin-Tate formal group $F$ of height $p^{h_{1} + h_{2}}$ when $p$ is odd.
     
     \begin{thm} \label{t4.4} Assume that $p$ is odd, let $\vec{h} = (h_{1}, h_{2})$ be a pair of positive integers that are relatively prime and both different from $1$, and let $F$ be the $2$-dimensional Lubin-Tate formal group associated with $\vec{h}$ by Definition \ref{defLTFGdim2}. Let $\DDD_{\vec{h}}$ be the $2$-dimensional dynamical system defined by \eqref{defDDD} above, which provides the simplest approximation of the multiplication-by-$p$ endomorphism of $F$. 
     
 Let $n \geq 1$ be any integer and $(\xi, \eta)$ be a non-trivial $p^{n}$-torsion point of $\DDD_{\vec{h}}$. Then the $p$-adic valuations of $\xi$ and $\eta$ are given by the following formulae, where $h := h_{1} + h_{2}$.
     \begin{itemize}
     \item If $n = 2m$ is even, then 
     \[\displaystyle v(\xi) = \frac{p^{h_{2}} + 1}{p^{hm - h_{1}}(p^{h} - 1)} \ \text{ and } \ v(\eta) = \frac{p^{h_{1}} + 1}{p^{hm - h_{2}}(p^{h} - 1)} \ . \]
     \item If $n =2m+1$ is odd, then 
     \[\displaystyle v(\xi) = \frac{p^{h_{1}} + 1}{p^{hm} (p^{h} - 1)} \ \text{ and } \ v(\eta) = \frac{p^{h_{2}} + 1}{p^{hm}(p^{h} - 1)} \ . \]
     \end{itemize}
     \end{thm}
    \begin{proof}
    We prove this result by (strong) induction on $n \geq 1$. In the sequel, {\it $p^{n}$-torsion point} always means non-trivial, i.e. that it lies in $F[p^{n}]$ but not in $F[p^{n-1}]$.
    \begin{itemize}
    \item Assume that $(\xi, \eta)$ is a $p$-torsion point of $\DDD_{\vec{h}}$, which means that we have 
    \[\displaystyle p\xi + \eta^{p^{h_{1}}} = 0 \ \text{ and } \ p\eta + \xi^{p^{h_{2}}} = 0 \ . \]
    The first equality shows that $p\xi = - \eta^{p^{h_{1}}}$, hence (by elevation to the $p^{h_{2}}$ power, which is odd) that $\displaystyle p^{p^{h_{2}}}\xi^{p^{h_{2}}} = - \eta^{p^{h_{1} + h_{2}}}$, while the second equality shows that $\xi^{p^{h_{2}}} = - p \eta$. We thus obtain that:
    \[ \begin{array}{rcl}
    & & p^{p^{h_{2}}+ 1}\eta = \eta^{p^{h}} \\
    & \text{i.e.} & \eta^{p^{h} - 1} = p^{p^{h_{2}} + 1} \\
    & \text{i.e.} & \eta = \zeta p^{\frac{p^{h_{2}} +1}{p^{h} - 1}}  \ , 
    \end{array} \]
    where $\zeta$ is some $(p^{h}-1)$-th root of $1$. This shows that $v(\eta)$ equals $\displaystyle \frac{p^{h_{2}} + 1}{p^{h} - 1}$, as claimed in the odd case (for $m = 0$) of the statement. Then, we obtain from $\xi^{p^{h_{2}}} = p\eta$ that
    \[\displaystyle v(\xi) = \frac{1 + v(\eta)}{p^{h_{2}}} = \frac{p^{h} + p^{h_{2}}}{p^{h_{2}}(p^{h} - 1)} = \frac{p^{h_{1}} + 1}{p^{h} - 1} \ , \]
as claimed in the odd case (for $m = 0$) of the statement: hence, the case $n = 1$ is proven.
    \item Now, let $(\xi, \eta)$ be a $p^{2}$-torsion point of $\DDD_{\vec{h}}$. Then $\DDD_{\vec{h}}(\xi, \eta) =: (\xi_{1}, \eta_{1})$ is a $p$-torsion point of $\DDD_{\vec{h}}$, thus it satisfies the valuation formulae we just proved for $n = 1$. Moreover, we have, by construction, that
    \[\displaystyle p\xi + \eta^{p^{h_{1}}} = \xi_{1} \ \text{ and } \ p\eta + \xi^{p^{h_{2}}} = \eta_{1} \ , \]
    which implies that 
    \begin{equation} \label{valuationp2torsion}
    v(\xi_{1}) = v(p\xi + \eta^{p^{h_{1}}}) \ \text{ and } v(\eta_{1}) = v(p\eta + \xi^{p^{h_{2}}}) \ . 
    \end{equation}
   Since the formulae obtained in the case $n = 1$ imply that $v(\xi_{1})$ and $v(\eta_{1})$ are both {\bf strictly} between $0$ and $1$ (as we assumed $h_{1}h_{2} \not=0$ and $p \geq 3$), we must have 
    \begin{equation} \label{valuationp2torsionmin}
    \displaystyle \left\{ \begin{array}{l} 
    v(\xi_{1}) = \min\{v(p\xi), v(\eta^{p^{h_{1}}})\} = \min\{1 + v(\xi), p^{h_{1}}v(\eta)\} = p^{h_{1}}v(\eta)  \text{ as } v(\xi) \geq 0 \\
    v(\eta_{1}) = \min\{v(p\eta), v(\xi^{p^{h_{2}}})\} = \min\{1 + v(\eta), p^{h_{2}}v(\xi)\} = p^{h_{2}}v(\xi) \text{ as } v(\eta) \geq 0 
    \end{array} \right. \ . 
    \end{equation}
 (Note that both $\xi$ and $\eta$ must have non-negative $p$-adic valuation since $\xi_{1}$ and $\eta_{1}$ do by the case $n=1$: indeed, one of them having negative valuation would imply that one of the minima reached in \eqref{valuationp2torsionmin} is negative, which is absurd.) This proves that 
 \[\displaystyle v(\eta) = p^{-h_{1}}v(\xi_{1}) = \frac{p^{h_{1}} + 1}{p^{h_{1}}(p^{h} -1)} \ \text{ and } v(\xi) = p^{-h_{2}}v(\eta_{1})= \frac{p^{h_{2}}+1}{p^{h_{2}}(p^{h} - 1)} \ , \]
 which are the formulae given by the statement of the theorem in the even case ($m = 1$).
    
    \item Finally, assume that the statement of the theorem holds for a given integer $N \geq 1$ and let us prove that it is true for $N+1$. Let $(\xi, \eta)$ be a $p^{N+1}$-torsion point of $\DDD_{\vec{h}}$ and set $(\xi_{N}, \eta_{N}) := \DDD_{\vec{h}}(\xi,\eta)$. Then $(\xi_{N}, \eta_{N})$ is a $p^{N}$-torsion point of $\DDD_{\vec{h}}$, thus it satisfies by assumption the valuation formulae of the theorem, which fall in either of the following cases.
    \begin{itemize}
    \item[\tt $\star$] If $N = 2m$ is even, then we have 
    \begin{equation} \label{valuationrecursioneven}
      \displaystyle   v(\xi_{N}) =  \frac{p^{h_{2}} + 1}{p^{hm - h_{1}}(p^{h} - 1)} \ \text{ and } v(\eta_{N}) = \frac{p^{h_{1}} + 1}{p^{hm - h_{2}}(p^{h} - 1)} \ .
    \end{equation}

    \item[\tt $\star$] If $N = 2m+1$ is odd, then we have 
    \begin{equation} \label{valuationrecursionodd}    
      \displaystyle   v(\xi_{N}) = \frac{p^{h_{1}} + 1}{p^{hm} (p^{h} - 1)} \ \text{ and } v(\eta_{N}) = \frac{p^{h_{2}} + 1}{p^{hm}(p^{h} - 1)} \ .
    \end{equation}
    \end{itemize}
    Moreover, we have, by construction, that
    \[\displaystyle p\xi + \eta^{p^{h_{1}}} = \xi_{N} \ \text{ and } \ p\eta + \xi^{p^{h_{2}}} = \eta_{N} \ . \]
Hence, the argument used to obtain the $p^{2}$-torsion case from the $p$-torsion case is directly transposable here to prove that $\displaystyle v(\xi_{N+1}) = p^{-h_{2}}v(\eta_{N}) \ \text{ and } \ v(\eta_{N+1}) = p^{-h_{1}} v(\xi_{N})$. We then just have to replace the values of $v(\xi_{N})$ and $v(\eta_{N})$, using either \eqref{valuationrecursioneven} or \eqref{valuationrecursionodd} (depending on the parity of $N$), to obtain the announced formulae for $N+1$.
    \end{itemize}
    This proves the theorem by induction principle.
    \end{proof}
    Let us point out that the first part of the proof of Theorem \ref{t4.4} provides the following result.
    \begin{cor} \label{ptorsionexpression}
    The non-trivial $p$-torsion points of $\DDD_{\vec{h}}$ are the points of coordinates 
    \[\displaystyle (\xi_{1}, \eta_{1}) = \left(\zeta_{h_{2}}\zeta_{h}^{p^{h_{1}}}p^{\frac{p^{h_{1}}-1}{p^{h}-1}} , \zeta_{h} p^{\frac{p^{h_{2}}+1}{p^{h}-1}}\right) \ , \]
    where $\zeta_{h}$ (resp. $\zeta_{h_{2}}$) runs over the $(p^{h}-1)$-th roots of $1$ (resp. $p^{h_{2}}$-th roots of $-1$).  
    \end{cor} 
\begin{proof}
The first step of the proof of Theorem \ref{t4.4} shows that there exists a $(p^{h}-1)$-th root of unity $\zeta_{h}$ such that $\eta_{1} = \zeta_{h}p^{\frac{p^{h_{2}}+1}{p^{h}-1}}$ and $\xi_{1}^{p^{h_{2}}} = p\eta_{1}$. Since $\zeta_{h}^{p^{h}} = \zeta_{h}$ with $h = h_{1} + h_{2}$, we have $\zeta_{h}^{p^{-h_{2}}} = \zeta_{h}^{p^{h_{1}}}$. Hence, extracting a $p^{h_{2}}$-th root of $\xi_{1}^{p^{h_{2}}} = p\eta_{1}$ provides the announced expression for $\xi_{1}$. Note that, conversely, each pair of the form of the statement provides a non-trivial $p$-torsion point, so the corollary is proven.
\end{proof}
     
An interesting consequence of these formulae is the following statement, which take care of the ramification properties of the extension generated by the (coordinates of the) $p$-torsion points of $F$ (or, equivalently, $\DDD_{\vec{h}}$) when $h_{1}$ and $h_{2}$ have distinct parities (i.e. when their sum is odd).
     
     \begin{thm}\label{t4.4bis} We keep the same notation as in Theorem \ref{t4.4} and we assume moreover that $h = h_{1} + h_{2}$ is odd. 
     Given a non-zero $p$-torsion point $(\xi_{1}, \eta_{1})$ of $\DDD_{\vec{h}}$, the following holds.
     \begin{enumerate}
     \item Both $\xi_{1}$ and $\eta_{1}$ generate the same extension of $U_{h}$, i.e., $U_{h}[\xi_{1}] = U_{h}[\eta_{1}]$.
     \item The extension $U_{h}[\xi_{1}]/U_{h}$ is totally ramified of degree $\displaystyle \frac{p^{h} -1}{2}$.
     \end{enumerate}
     \end{thm}

\begin{proof} 
Since we assume $h = h_{1} + h_{2}$ odd, the assumptions of Lemma \ref{l4.3} are satisfied by $(h_{1}, h)$ and $(h_{2},h)$. Hence, $\displaystyle \frac{p^{h} - 1}{2}$ and $\displaystyle \frac{p^{h_{1}} + 1}{2}$ are relatively prime, as well as $\displaystyle \frac{p^{h} - 1}{2}$ and $\displaystyle \frac{p^{h_{2}} + 1}{2}$. This ensures that for any index $i \in \{1,2\}$, the polynomial $P_{i}(T) := T^{\frac{p^{h}-1}{2}} - p^{\frac{p^{h_{i} }+1}{2}}$ is irreducible over $U_{h}$, since its degree is coprime with the $p$-adic valuation of its constant coefficient and its Newton polygon consists in a single line segment (see \cite[Theorem 3.2]{LEF}). By double inclusion, this implies that
\begin{equation} \label{totram1}
\forall \ i \in \{1, 2\}, \ U_{h}\left(p^{\frac{2}{p^{h}-1}}\right) = U_{h}\left(p^{\frac{(p^{h_{i}}+1)/2}{(p^{h}-1)/2}} \right) = U_{h}\left(p^{\frac{p^{h_{i}}+1}{p^{h}-1}} \right) \ ,
\end{equation}
thus we have $U_{h}(\xi_{1}) = U_{h}(\eta_{1})$ by Corollary \ref{ptorsionexpression}, and statement (1) is proven.

As we saw above, $P_{1}(T)$ is irreducible over $U_{h}$ and vanishes at $\xi_{1}$: hence, it is the minimal polynomial of $\xi_{1}$ and $U_{h}(\xi_{1})$ is a rupture field\footnote{Let us recall that a {\it rupture field} for an irreducible polynomial over $U_{h}$ is an extension of $U_{h}$ that is generated by a root (in $\QQ_{p}$) of this polynomial. In general, it is not a splitting field for the polynomial, as it may not contain all the roots (in $\QQ_{p}$) of the polynomial.} for $P_{1}$ over $U_{h}$. This ensures directly that 
\[ \displaystyle \left[U_{h}(\xi_{1}) : U_{h}\right] = \deg(P_{1}) = \frac{p^{h}-1}{2} \ . \]
Finally, let $e$ denote the ramification index of $U_{h}(\xi_{1})$ over $U_{h}$, so that we have $v(\varpi) = e^{-1}$ for any uniformizer $\varpi$ of $U_{h}(\xi_{1})$. Note that the $\varpi$-adic valuation of $\lambda := p^{\frac{2}{p^{h}-1}}$, which belongs to $U_{h}(\xi_{1})$ by \eqref{totram1}, is then equal to $\frac{2e}{p^{h}-1}$. Since it must be a non-negative integer, this proves that $2e$ is a multiple of $p^{h}-1$, hence that $e$ is divisible by $\frac{p^{h}-1}{2} = [U_{h}(\xi_{1}) : U_{h}]$. This shows that $e = [U_{h}(\xi_{1}) : U_{h}]$, i.e. that $U_{h}[\xi_{1}]/U_{h}$ is totally ramified, and statement (2) is proven.
\end{proof}

\begin{rem}
The study of the ramification properties of the extension generated by the (coordinates of the) $p^{n}$-torsion points when $n > 1$ is the purpose of a forthcoming work \cite{AS2}.
\end{rem}

	\hrulefill
	\appendix 

\section{An alternative proof of Proposition \ref{p1.1}} \label{sb}
The goal of this appendix is to provide a self-contained proof of Proposition \ref{p1.1} for $2$-dimensional formal groups over $\Z_{p}$, then to give some intuition to the reader by providing some heuristic on a concrete example (see Example \ref{exsb} below).

Using the same notation as in Proposition \ref{p1.1}, we set 
\[\displaystyle L_{k}(X) := p^{-k}[p^{k}]_{F}(X) \ \text{ for any integer } k \geq 1 \ .\]
To prove that the $p$-adic power series $\displaystyle L := \lim_{k \to +\infty} L_{k}$ exists, it suffices to prove that $\{L_{k}\}_{k \geq 1}$ is a Cauchy sequence (since the field $\Q_{p}$ of $p$-adic numbers is complete for the $p$-adic norm). We do this following the idea of the proof of \cite[Lemma 3]{AW}. 

Let us first note that for any integers $m \geq n \geq 1$, we have 
\begin{equation} \label{A1Cauchy}
L_{m}(X) - L_{n}(X) = p^{-m} \left([p^{m-n}]_{F} - p^{m-n} \mathrm{Id}_{F} \right) \circ [p^{n}]_{F}(X) \ . 
\end{equation}

Let us also fix some notation: for any pair $I = (i_{1}, i_{2})$ of positive integers, we set $\ell(I) := i_{1} + i_{2}$. For any integers $k,\ell$ satisfying $0 \leq k \leq \ell$, we set $X_{k}^{\ell} := \begin{pmatrix} x_{1}^{\ell-k} x_{2}^{k} \\ x_{1}^{k} x_{2}^{\ell - k} \end{pmatrix}$. Then, we have
\begin{equation} \label{decompomatrix}
\displaystyle [p]_{F}(X) = \sum_{\ell \geq 1} \sum_{k = 0}^{\lfloor \frac{\ell}{2}\rfloor} C_{k}^{\ell} X_{k}^{\ell} \ \text{ for some diagonal matrices } C_{k}^{\ell} \in \mathrm{M}_{2}(\Z_{p}) \ , 
\end{equation}
 and we can moreover assume\footnote{Here, it looks like an assumption, but Proposition \ref{congruencesdim2} shows that this is actually the only reasonable choice.} that $C_{0}^{1} = p\mathrm{I}_{2}$.
 
We will measure the ``size'' of the components of this series by the mean of the function $\textrm{val}$ defined as follows : 
\begin{equation}
\label{defval}
\forall \ \ell \geq 1, \forall \ k \in \left\{0, \ldots , \lfloor \frac{\ell}{2} \rfloor \right\} \ , \textrm{val}(C_{k}^{\ell} X_{k}^{\ell}) := v(C_{k}^{\ell}) + \ell \ ,  
\end{equation}
where $v$ denotes the usual extension of the $p$-adic valuation to $\Q_{p}$-valued vectors and matrices (given by taking the infimum of the $p$-adic valuation of the coefficients of the vectors and matrices). Note that these values are all non-negative, as the matrices $C_{k}^{\ell}$ have all their coefficients in $\Z_{p}$ and $\ell$ is positive, and that we have 
\[\displaystyle \forall \ \ell \geq 2, \forall \ k \in \left\{ 0, \ldots , \lfloor \frac{\ell}{2} \rfloor \right\}, \ \textrm{val}(C_{k}^{\ell} X_{k}^{\ell}) \geq \ell \geq 2 = v(p) + 1 = \textrm{val}(C_{0}^{1} X_{0}^{1}) \ .\]
Given any endomorphism $\alpha$ of $F$ over $\Z_{p}$, we let $\mathrm{val}(\alpha)$ be the minimum value reached by $\mathrm{val}$ on the monomial appearing in the decomposition of $\alpha$ under the form \eqref{decompomatrix}. For instance, we have just proven that $\mathrm{val}([p]_{F}(X)) = 2$, and the recursion formula provided by Example \ref{defmultiplicationmaps} then imply that we have:
\begin{equation} \label{boundonval}
\displaystyle \forall \ n \geq 1, \ \mathrm{val}([p^{n}]_{F}(X)) \geq n + 1 \ . 
\end{equation}
Let us write $[p^{m-n}]_{F}$ under the form given by \eqref{decompomatrix} : 
\[ \displaystyle [p^{m-n}]_{F}(X) = \sum_{\ell \geq 1} \sum_{k = 0}^{\lfloor \frac{\ell}{2}\rfloor} M_{k}^{\ell} X_{k}^{\ell} \ \text{ for some diagonal matrices } M_{k}^{\ell} \in \mathrm{M}_{2}(\Z_{p}) \ . \]
The lower bound given by \eqref{boundonval} then ensures that:
\[\forall \ \ell \geq 1, \forall \ k \in \left\{ 0, \ldots , \lfloor \frac{\ell}{2} \rfloor \right\}, \ \textrm{val}(M_{k}^{\ell}X_{k}^{\ell}) \geq m-n + 1 \ , \]
which implies that all coefficients of $M_{k}^{\ell}$ are divisible by $p^{m-n+1 - \ell}$. Let $\ell_{0}$ be the smallest index for which there exists $k \in \left\{ 0, \ldots , \lfloor \frac{\ell}{2} \rfloor \right\}$ such that $\mathrm{val}([p^{m-n}]_{F}) = \mathrm{val}(M_{k}^{\ell_{0}}X_{k}^{\ell_{0}}) \geq m -n + 1 - \ell_{0}$.

Now observe that for any integer $k \geq 1$, the first term of $[p^{k}]_{F}(X)$ is $p^{k}X$, i.e. that 
\[\displaystyle [p^{k}]_{F}(X) - p^{k}X \equiv 0 \text{ mod degree $2$ terms.} \]
Using this for $k = 1$ and for $k = m-n$, we then obtain that all monomials $M_{k}^{\ell}X^{\ell}_{k}$ involved in $[p^{m-n}]_{F}(X) - p^{m-n}X$ must satisfy $\ell \geq {\color{black}{2}}$. We hence obtain that $\left([p^{m-n}]_{F} - p^{m-n}\mathrm{Id}_{F}\right) \circ [p^{n}]_{F}(X)$ must satisfy 
\[\displaystyle \mathrm{val}\left(\left([p^{m-n}]_{F} - p^{m-n}\mathrm{Id}_{F}\right) \circ [p^{n}]_{F}(X) \right) \geq {\color{black}{2}}(n+1) + (m-n+1 - 2) = m + n + 1 \ . \]
Thanks to \eqref{A1Cauchy}, this finally proves that 
\[ \displaystyle \mathrm{val}\left(L_{m}(X) - L_{n}(X)\right) \geq (m+n+1) - m = n + 1 \ , \]
which goes to $+ \infty$ as $n$ does. This ensures that the $p$-adic norm of any coefficient of $L_{m}(X) - L_{n}(X)$ goes to $0$ as $n$ goes to infinity, hence, that $\{L_{k}\}_{k \geq 1}$ is a Cauchy sequence, as claimed. We can thus define $\displaystyle L(X) = \lim_{k \to \infty} L_{k}(X)$ and, since we have $J(L_{k}(X)) = \mathrm{I}_{2}$ for any $k \geq 1$, we obtain that $J(L(X)) = \displaystyle \lim_{k \to \infty} J(L_{k}(X)) $ is also equal to $\mathrm{I}_{2}$, so we are done. 

\begin{exmp} \label{exsb}
To help the reader getting a better understanding of the proof above, let us give a concrete example. Consider the case where $(h_{1}, h_{2}) = (4,5)$: hence, our $2$-dimensional Lubin-Tate formal group $F$ is of height $9$. Note that we are in a case where $h_{1}$ and $h_{2}$ are both greater than $2$ and coprime.

In this case, the first non-zero terms of the multiplication-by-$p$ endomorphism $[p]_{F}(X)$ can for instance be taken as follows: 
\[\displaystyle \begin{array}{rlll}
C_{0}^{1} = \begin{pmatrix} p & 0 \\ 0 & p \end{pmatrix} \ ; & C^{p^{2}}_{0} = \begin{pmatrix} p^{2} & 0\\ 0 & p^{3} \end{pmatrix} \ ; & C_{p}^{p^{2} + p} = \begin{pmatrix} p^{4} & 0\\ 0 & p^{3} \end{pmatrix} \ ; & \\
C_{p^{3}}^{p^{3} + p} = \begin{pmatrix} 0 & 0 \\ 0 & p \end{pmatrix} \ ; & C_{0}^{p^{4}} = \begin{pmatrix} 0 & 0 \\ 0 & 1 \end{pmatrix} \ ; & C_{p^{4}}^{p^{4} + p} = \begin{pmatrix} p & 0 \\ 0 & 0  \end{pmatrix} \ ; & C_{0}^{p^{5}} = \begin{pmatrix} 1 & 0 \\ 0 & 0 \end{pmatrix}.
\end{array}\]
One can then check that  
\[ \displaystyle [p]_F(X) \equiv \begin{pmatrix}
px_1 \\ px_2
\end{pmatrix}~(\text{mod degree 2 terms})~\text{and}~[p]_F(X) \equiv \begin{pmatrix}
x_2^{p^4} \\ x_1^{p^5}
\end{pmatrix}~(\text{mod $p$}) \ ,  \]
which is consistent with the formulae given later by Proposition \ref{congruencesdim2}.
	\end{exmp}
	
		\section{Beyond the $2$-dimensional case : some obstructions} \label{sa}
Constructing Lubin-Tate formal groups of dimension $d \geq 3$ requires to overcome the three following challenges, at least when the motivation is to use them to extend the original applications of Lubin and Tate to class field theory \cite{JL2}.
		\begin{itemize}
		\item For now, the most developed tools to define and study arbitrary dimensional formal groups over $\Z_{p}$ (or, more generally, the ring of integers of a finite extension of $\Q_{p}$) are provided by Hazewinkel's work \cite{MH}. This is why we used them to construct our $2$-dimensional formal groups following Lubin-Tate approach, which constructs a formal group from its logarithm, the latter being defined using Hazewinkel's functional equation (as we did in Section \ref{s2}) for a carefully chosen set of $2 \times 2$-matrices $(s_{i})_{i \geq 0}$.
	
	Going to higher dimension $d \geq 3$ requires to select an appropriate set of $d \times d$-matrices $(s_{i})_{i \geq 0}$ and to manage to handle the recursion formula \eqref{e2} associated with this set of data to define the logarithm of the $d$-dimensional Lubin-Tate formal group to be: this is definitely much more complex and technical than the $2$-dimensional case.
		\item Another technical challenge when going to higher dimension is to identify the appropriate candidate for the endomorphism $[p]_{F}$ defined by the multiplication by $p$. In the $2$-dimensional case, we have $2$ candidates coming from \eqref{defCLThl}, associated with the two possible permutations of $\{p^{h_{1}}, p^{h_{2}}\}$ as exponents of $x_{1}$ and $x_{2}$. Then, \eqref{e6} allows us to show that exactly one of them was appropriate.
		
		When going to dimension $d \geq 3$, we face a priori $d!$ candidates for the endomorphism $[p_{F}]$ (corresponding to the possible permutations of exponents for $x_{1}, \ldots , x_{d}$ given by $(p^{h_{i}})_{1 \leq i \leq d}$, where $h = (h_{i})_{1 \leq i \leq d}$ is the prescribed height of the Lubin-Tate formal group to be), and checking which options actually provide endomorphisms of $F$ becomes a much more tidious task as $d$ grows.
		
		\item A third obstruction, which seems more serious when $d$ grows, is to understand what should be the $p$-adic geometry of a $d$-dimensional Lubin-Tate formal group to have enough conditions ensuring their existence. For instance, let us recall that in the $2$-dimensional case, we introduced an appropriate notion of Newton copolygon for $p$-adic power series in $2$ variables, which lives in the $3$-dimensional Euclidean space, to obtain information on the $p$-adic geometry of our Lubin-Tate formal groups (see for instance Example \ref{e3.2}) and on the ramification properties of the extension generated by the $p^{\infty}$-torsion points (see for instance Theorem \ref{t4.4bis}). In this case, things remain under control in dimension $3$ since what we have to examinate is intersections of curves or surfaces with planes and lines.
		
		Going to dimension $d \geq 3$ would then require to define the Newton copolygon associated with a $p$-adic power series in $d$-variables, which would live in the $d+1$-dimensional Euclidean space, and to examinate intersections of $k$-dimensional surfaces in this space with $\ell$-dimensional sub-vectorspaces: this already represents a much more complex task, and illustrate the lack of tools we have actually at hand to explore the $p$-adic geometry of possible Lubin-Tate formal groups of dimension $d \geq 3$.
				\end{itemize}
\end{document}